\documentclass[11pt,a4paper,reqno]{amsart}
\usepackage[english]{babel}
\usepackage[T1]{fontenc}
\usepackage{verbatim}
\usepackage[shortlabels]{enumitem}
\usepackage{palatino}
\usepackage{amsmath}
\usepackage{mathabx}
\usepackage{amssymb}
\usepackage{amsthm}
\usepackage{amsfonts}
\usepackage{graphicx}
\usepackage{esint}
\usepackage{color}
\usepackage{mathtools}
\usepackage{overpic}

\usepackage[colorlinks = true, citecolor = black]{hyperref}
\author{Tuomas Orponen}
\title{Planar sets with large visible parts}
\address{Department of Mathematics and Statistics\\ University of Jyv\"askyl\"a,
P.O. Box 35 (MaD)\\
FI-40014 University of Jyv\"askyl\"a\\
Finland}
\email{tuomas.t.orponen@jyu.fi}
\date{\today}
\thanks{T.O. is supported by the European Research Council (ERC) under the European Union’s Horizon Europe research and innovation programme (grant agreement No 101087499)}

\newcommand{\R}{\mathbb{R}}

\newcommand{\N}{\mathbb{N}}

\newcommand{\Z}{\mathbb{Z}}

\newcommand{\Hd}{\dim_{\mathrm{H}}}

\newcommand{\diam}{\operatorname{diam}}
\newcommand{\card}{\operatorname{card}}
\newcommand{\dist}{\operatorname{dist}}

\newcommand{\vis}{\mathrm{Vis}}

\def\Barint_#1{\mathchoice
          {\mathop{\vrule width 6pt height 3 pt depth -2.5pt
                  \kern -8pt \intop}\nolimits_{#1}}%
          {\mathop{\vrule width 5pt height 3 pt depth -2.6pt
                  \kern -6pt \intop}\nolimits_{#1}}%
          {\mathop{\vrule width 5pt height 3 pt depth -2.6pt
                  \kern -6pt \intop}\nolimits_{#1}}%
          {\mathop{\vrule width 5pt height 3 pt depth -2.6pt
                  \kern -6pt \intop}\nolimits_{#1}}}

\numberwithin{equation}{section}

\theoremstyle{plain}
\newtheorem{thm}[equation]{Theorem}
\newtheorem{"thm"}[equation]{"Theorem"}

\newtheorem{lemma}[equation]{Lemma}
\newtheorem{"lemma"}[equation]{"Lemma"}

\newtheorem{"proposition"}[equation]{"Proposition"}
\newtheorem{proposition}[equation]{Proposition}

\newtheorem{problem}{Problem}

\theoremstyle{definition}

\newtheorem{definition}[equation]{Definition}

\newtheorem{notation}[equation]{Notation}

\theoremstyle{remark}
\newtheorem{remark}[equation]{Remark}

\newcommand{\nref}[1]{(\hyperref[#1]{#1})}

\DeclareMathSymbol{\intop}  {\mathop}{mathx}{"B3}

\begin{document}

\begin{abstract} I construct a compact subset of the plane whose visible parts are $\tfrac{3}{2}$-dimensional in all directions. This disproves the visibility conjecture. The value $\tfrac{3}{2}$ cannot be increased, as shown in recent collaboration with A. Rutar.  \end{abstract}

\maketitle

\tableofcontents

\section{Introduction}

This paper studies the visible parts of compact sets in the plane. We start by defining relevant notation and terminology.
\begin{notation}\label{not1}
    For $z \in \R^{2}$ and $\sigma \in [0,1)$, let 
    \begin{displaymath} \ell^{+}_{z,\sigma} \coloneqq \{z + (r,\sigma r) : r \geq 0\} \quad \text{ and } \quad \ell_{z,\sigma} \coloneqq \{z + (r,\sigma r) : r \in \R\}. \end{displaymath}
    These are the \emph{ray} and \emph{line with slope $\sigma$ passing through $z$}.
\end{notation}

\begin{definition}[$\vis_{\sigma}(K)$ and $\vis^{2}_{\sigma}(K)$]\label{def:visiblePart}
    Let $K \subset \R^{2}$ be a set, and $\sigma \in [0,1)$.
    The \emph{visible part of $K$ in direction $\sigma$} is the set
    \begin{equation*}
        \vis_{\sigma}(K) \coloneqq \{z \in K : K \cap \ell^{+}_{z,\sigma} = \{z\}\}.
    \end{equation*}
    The \emph{bi-visible part of $K$ in direction $\sigma$} is the set
    \begin{displaymath} \vis_{\sigma}^{2}(K) := \{z \in K : K \cap \ell_{z,\sigma} = \{z\}\}. \end{displaymath}
\end{definition}
\begin{remark} Evidently $\vis^{2}_{\sigma}(K) \subset \vis_{\sigma}(K)$ for all $\sigma \in [0,1)$. This inclusion can be strict for all $\sigma \in [0,1)$. For example $\vis_{\sigma}(S^{1})$ is an arc for all $\sigma \in [0,1)$, whereas $\card \vis_{\sigma}^{2}(S^{1}) \equiv 2$.
\end{remark} 

The \emph{visibility problem} in fractal geometry is the quest for finding upper bounds on the Hausdorff dimension of $\vis_{\sigma}(K)$, typically when $K$ is compact. The first positive results were published by J\"arvenp\"a\"a, J\"arvenp\"a\"a, MacManus, and O'Neil \cite{zbl:1026.28002} in 2003. They showed that if $K \subset \R^{2}$ is (i) a quasi-circle or (ii) a connected self-similar set without rotations, then $\Hd \vis_{\sigma}(K) \leq 1$ for all $\sigma \in [0,1)$. They also showed that the visible parts of graphs $\Gamma \subset \R^{2}$ of continuous functions $[0,1] \to [0,1]$ are at most $1$-dimensional in all directions, except possibly one: the (bi-)visible part of $\Gamma$ in the vertical direction is evidently $\Hd \Gamma$-dimensional, and $\Hd \Gamma$ can take any value in $[1,2]$.

While no "visibility conjecture" was formulated in \cite{zbl:1026.28002}, it soon turned into common belief that if $K \subset \R^{2}$ is compact, then $\Hd \vis_{\sigma}(K) \leq 1$ for almost all $\sigma \in [0,1)$. This has been explicitly suggested (at least) in \cite[Problem~11]{zbl:1049.28007}, \cite[Conjecture~1.3]{zbl:1267.28006}, \cite[Section 10]{MR3558147}, and \cite[Section~7]{arxiv:2602.22002}. The hypothesis has been confirmed for certain special families of sets such as quasi-circles, fractal percolation, and some self-similar and self-affine sets (see \cite{zbl:1251.28007,zbl:1267.28006,zbl:1026.28002,zbmath:7541839,zbl:1479.28008}).
There is also a strong partial result for planar continua due to O'Neil \cite{zbl:1152.28318}.
In \cite{zbl:1119.28003}, it is shown that if $s \in (1,2]$, and $K \subset \R^{2}$ is a compact set with $\mathcal{H}^{s}(K) < \infty$, then $\mathcal{H}^{s}(\vis_{\sigma}(K)) = 0$ for a.e.
$\sigma \in [0,1)$.

For general compact sets, the first partial result appeared in \cite{zbl:1561.28083}, and this first a.e. upper bound ($1.99$) was subsequently lowered by Matheus and D{\k a}browski \cite{zbmath:8092718,zbl:1484.28010}. In 2026, we improved the bound further with A. Rutar \cite{2026arXiv260606965O} by establishing the following:
\begin{thm}\label{thm:OR} Let $K \subset \R^{2}$ be compact. Then $\Hd \vis_{\sigma}(K) \leq \tfrac{3}{2}$ for a.e. $\sigma \in [0,1)$. \end{thm} 

In this paper I prove that Theorem \ref{thm:OR} is sharp:
\begin{thm}\label{main} There exists a compact set $K \subset \R^{2}$ such that $\Hd \vis^{2}_{\sigma}(K) \geq \tfrac{3}{2}$ for all $\sigma \in [0,1)$. In particular $\Hd \vis_{\sigma}(K) \geq \tfrac{3}{2}$ for all $\sigma \in [0,1)$. \end{thm}

\begin{remark} The main point in Theorem \ref{main} lies in the corollary about $\vis_{\sigma}(K)$. The result for $\vis_{\sigma}^{2}(K)$ however had no extra cost, and it sheds light on a problem proposed in \cite[Section 6.4]{MR3617376}: for Lebesgue typical $\sigma \in [0,1)$, how much of a compact set $K \subset \R^{2}$ can be covered by "$s$-light" lines $\ell \subset \R^{2}$ with slope $\sigma$, satisfying $\Hd (K \cap \ell) \leq s$, where $s < \Hd K - 1$? Theorem \ref{main} shows that a $\tfrac{3}{2}$-dimensional set can be covered by such "$s$-light" lines, for every $\sigma \in [0,1)$, and every $s > 0$.  \end{remark} 

\begin{remark} Often visible parts are indexed by directions $e \in S^{1}$ instead of $\sigma \in [0,1)$. I thank E. J\"arvenp\"a\"a for pointing out that Theorem \ref{main} implies the existence of a compact set $K \subset \R^{2}$ whose (bi-)visible parts are $\tfrac{3}{2}$-dimensional for all $e \in S^{1}$. In fact, a union of four rotated and suitably translated copies of the set $K_{0}$ in Theorem \ref{main} has the desired property ("$4$" comes from the fact that the slopes in $[0,1)$ correspond to directions $e \in S^{1}$ making an angle in $[0,\pi/4)$ with the positive $x$-axis.) Precisely, let $R_{\varphi}$ be a the rotation by angle $\varphi \in [0,\pi)$. Then, assuming $K_{0} \subset B(0,\tfrac{1}{4})$ (as we may by dilating), a set of the form
\begin{displaymath} K := (K_{0} + w_{0}) \cup (R_{\pi/4}(K_{0}) + w_{1}) \cup (R_{\pi/2}(K_{0}) + w_{2}) + (R_{3\pi/4}(K_{0}) + w_{3}) \end{displaymath}
has the desired property. There is plenty of freedom in choosing the translation vectors $w_{0},\ldots,w_{3}$. One only needs to ensure that lines passing through $B(w_{j},\tfrac{1}{4})$ with angle in $[j\pi/4,(j + 1)\pi/4)$ do not intersect the discs $B(w_{i},\tfrac{1}{4})$ with $i \in \{0,1,2,3\} \, \setminus \, \{j\}$. One concrete choice is given by $(w_{0},w_{1},w_{2},w_{3}) = ((0,0),(-1,2),(-2,0),(2,3))$.  \end{remark}

\begin{remark} There is a heuristic way to justify why "$\tfrac{3}{2}$" is the right answer to the visibility problem. Upon any close inspection this heuristic cannot stand daylight, so I only add it here for entertainment. Fix $s \in [1,2]$, and pick a compact set $K \subset \R^{2}$ with $\Hd K = s$. Up to rotation, the visible parts $\vis_{\sigma}(K)$ are graphs of functions $f_{\sigma} \colon [0,1] \to [0,1]$. If it happened that each $f_{\sigma}$ is $\alpha$-H\"older continuous for some $\alpha \in [0,1]$, what might be the largest possible $\alpha$ in terms of $s$? After a moment of thought, the value $\alpha = s - 1$ may start to seem plausible (think of the cases $s \in \{1,2\}$). Assuming that this is correct, let us ask: what is the maximal Hausdorff dimension of the graph of an $(s - 1)$-H\"older function? The answer is $3 - s$, see \cite[Section, Theorem 6]{MR833073}. Therefore the visible parts of $K$ might conceivably be $(3 - s)$-dimensional. On the other hand, $\Hd \vis_{\sigma}(K) \leq \Hd K = s$. With this in mind, it might be possible to construct an $s$-dimensional compact set $K \subset \R^{2}$ with $\Hd \vis_{\sigma}(K) = \min\{s,3 - s\}$ for all $\sigma \in [0,1)$. Now note that $\min\{s,3 - s\} \leq \tfrac{3}{2}$ for all $s \in [1,2]$, with equality exactly if $s = \tfrac{3}{2}$.

The above is just a heuristic. In particular, I do not know if the set $K$ in Theorem \ref{main} is (or can be constructed to be) $\tfrac{3}{2}$-dimensional. This seems plausible, but would (at least) require significant additional work. \end{remark} 

While Theorems \ref{thm:OR} and \ref{main} resolve the visibility problem as commonly stated in the literature, the problem remains open in the connected and Ahlfors regular cases:

\begin{problem}\label{problem2} Assume that $K \subset \R^{2}$ is compact and (a) connected, or (b) Ahlfors $s$-regular with $s \in [0,2]$. Is it true that $\Hd \vis_{\sigma}(K) \leq \min\{1,s\}$ for a.e. $\sigma \in [1,2]$? \end{problem}

\begin{remark} The set constructed in Theorem \ref{main} is far from Ahlfors regular (let alone connected). There is partial progress towards Problem \ref{problem2}: D{\k a}browski has proved that if $K \subset \R^{2}$ is compact and Ahlfors $s$-regular with $s \in [1,2]$, then $\Hd \vis_{\sigma}(K) \leq s - a(s - 1)$ for a.e.
$\sigma \in [-1,1]$, where $a > 0.183$ is absolute. In particular, a.e. visible parts of Ahlfors $\tfrac{3}{2}$-regular sets are strictly less than $\tfrac{3}{2}$-dimensional.  

Remarkably, Problem \ref{problem2} is open for general self-similar sets, even without rotations. \end{remark}

The construction in Theorem \ref{main} also leaves open the following "endpoint" problem:

\begin{problem} Does there exist a compact set $K \subset \R^{2}$ such that $\mathcal{H}^{3/2}(\vis_{\sigma}(K)) > 0$ for all $\sigma \in [0,1)$? By the results in \cite{zbl:1119.28003}, this is only conceivable if $\mathcal{H}^{3/2}(K) = \infty$.  \end{problem}

The visibility problem in higher dimensions remains open. Many of the partial results cited above Theorem \ref{thm:OR} are already stated (in possibly weaker form) in all dimensions.

\subsection{Paper outline} Section \ref{s2} contains $\delta$-discretised statements (the main result there is Theorem \ref{thm1}) which are used to prove Theorem \ref{main} in Section \ref{s3}. Appendix \ref{appA} contains the construction of a "basic building block" behind the construction; the details of this object are carefully explained in Section \ref{s2}.

\subsection*{Notation} We use the convention that $\mathbb{N} = \{0,1,2,\ldots\}$. For $\delta \in 2^{-\N}$ and $A \subset \R^{d}$, the notation $\mathcal{D}_{\delta}(A)$ refers to the (standard) dyadic $\delta$-cubes of $\R^{d}$ which intersect $A$. We also write $\mathcal{D} := \bigcup_{\delta \in 2^{-\N}} \mathcal{D}_{\delta}(\R^{2})$ (note that we only include in $\mathcal{D}$ dyadic cubes of side-length $\leq 1$). The side-length of $Q \in \mathcal{D}$ is denoted $\ell(Q)$. We abbreviate $\mathcal{D}_{\delta} := \mathcal{D}_{\delta}([0,1)^{2})$. 

The notation $\mathcal{H}^{t}_{\infty}$ will refer to the \emph{dyadic $t$-dimensional Hausdorff content}:
\begin{equation}\label{def:content} \mathcal{H}^{t}_{\infty}(K) = \inf \left\{ \sum_{i} \ell(Q_{i})^{t} : K \subset \bigcup_{i} Q_{i} \right\}, \qquad K \subset \R^{d}, \end{equation}
where the "$\inf$" runs over all families $\{Q_{i}\} \subset \mathcal{D}$. We abbreviate $\mathcal{H}^{t}_{\infty}(\mathcal{P}) := \mathcal{H}^{t}_{\infty}(\cup \mathcal{P})$ for $\mathcal{P} \subset \mathcal{D}$. If $\mathcal{P}$ consists of dyadic cubes of side-length $\geq \delta \in 2^{-\N}$, and $t \in [0,d]$, it is easy to check (and will be implicitly used) that the optimal cover in the definition of $\mathcal{H}^{t}_{\infty}(\mathcal{P})$ also consists of dyadic cubes side-length $\geq \delta$. This boils down to the fact that if $\{\mathcal{Q}_{j}\} \subset \mathcal{D}$ is an arbitrary cover of $Q \in \mathcal{D}_{\delta}$ by dyadic cubes, then $\sum_{j} \ell(Q_{j})^{t} \geq \ell(Q)^{t}$.

We will also use the following notational convention: if $\mathcal{Q} \subset \mathcal{D}_{\Delta}$ and $\mathcal{P} \subset \mathcal{D}_{\delta}$ with $\delta \leq \Delta$, we will write 
\begin{displaymath} \mathcal{Q} \cap \mathcal{P} := \{p \in \mathcal{P} : p \subset Q \text{ for some } Q \in \mathcal{Q}\}. \end{displaymath} 
For $\delta > 0$, a \emph{$\delta$-tube} refers to the closed $(\delta/2)$-neighbourhood of a line in $\R^{2}$.

The \emph{slope} of a non-vertical line $\ell := \{(x,y) : y = ax + b\} \subset \R^{2}$ is $\sigma(\ell) := a$.

\section{Discretised results}\label{s2}

In this section we perform a number of $\delta$-discretised constructions to help us find the compact set $K$ in Theorem \ref{main} (the final construction can be found in Section \ref{s3}). A key component is the "basic building block" (BBB) in Proposition \ref{prop1}. Informally speaking, the BBB is a compact set $K \subset [0,1]^{2}$ with $\Hd K = \tfrac{3}{2}$ and the following property: for every slope $\sigma \in [0,1)$, there exists a $\tfrac{1}{2}$-dimensional family of lines with slope $\sigma$ which "pass near $K$ but do not touch $K$". For more precision on this mysterious property, see Proposition \ref{prop1}(2)-(3).

The BBB is closely related to another, well-known, construction in projection theory: there exists a compact set $K \subset \R^{2}$ such that $\Hd K = \tfrac{3}{2}$ and 
\begin{equation}\label{form23} \Hd \{\sigma \in [0,1) : \mathcal{H}^{1}(\pi_{\sigma}(K)) = 0\} = \tfrac{1}{2}. \end{equation}
Here $\pi_{\sigma}(x,y) = \sigma x + y$. As far as I know, a set $K$ like this was first constructed by Peltom\"aki \cite{Peltomaki} in the late 80s. I will however need a more recent (and $\delta$-discretised) incarnation presented in \cite[Appendix A.1]{MR4745881}. In fact, $K$ is nothing but a "$\tfrac{3}{2}$-dimensional grid" and the slopes $\sigma \in [0,1)$ appearing in \eqref{form23} come from a "$\tfrac{1}{2}$-dimensional arithmetic progression". It is no coincidence that the pair $(\tfrac{3}{2},\tfrac{1}{2})$ appears both in the description of the BBB, and in \eqref{form23}: the BBB is obtained by suitably applying point-line duality to the set underlying \eqref{form23}. The details are so lengthy that they are postponed to Appendix \ref{appA}. However, a depiction of the BBB is shown in Figure \ref{fig1}.

The main purpose of this section is to apply the BBB to build something that can be directly used to prove Theorem \ref{main}. The main result of the section is Theorem \ref{thm1}. Roughly speaking, Theorem \ref{thm1} contains the construction of (nearly) $\tfrac{3}{2}$-dimensional families $\mathcal{P}_{\theta} \subset \mathcal{D}_{\delta}$, indexed by $\theta \in \mathcal{D}_{\delta}([0,1))$, with the following property: $\mathcal{P}_{\theta}$ cannot "block the view" to $\mathcal{P}_{\theta'}$ along lines parallel to $\theta$ (or $\theta'$), except if $\theta,\theta'$ are "neighbours".

 The families $\mathcal{P}_{\theta}$ in Theorem \ref{thm1} are obtained by iterating the BBB many times. All direct communication with the BBB happens within Lemma \ref{prop2}, which iterates the BBB many times. Then the proof of Theorem \ref{thm1} iterates Lemma \ref{prop2} a few times more.

\begin{proposition}[Basic building block]\label{prop1} For every $\epsilon,\Delta \in 2^{-\N}$, there exists $\delta_{0} = \delta_{0}(\epsilon,\Delta) \in (0,\tfrac{1}{2}\Delta]$ such that the following holds for all $\delta \in 2^{-\N} \cap (0,\delta_{0}]$. There exists a family $\mathfrak{P} \subset \mathcal{D}_{\delta}$ with the following property. Assume that $t \in [1 + \epsilon,\tfrac{3}{2}]$, and $\mathcal{Q} \subset \mathcal{D}_{\Delta}$. Then 
\begin{equation}\label{form3} \mathcal{H}_{\infty}^{t - \epsilon}(\mathcal{Q} \cap \mathfrak{P}) \geq \mathcal{H}^{t}_{\infty}(\mathcal{Q}). \end{equation}
Moreover, for each $\sigma \in \delta \Z \cap [0,1)$ there exists a family $\mathcal{R}_{\sigma}$ of $(\delta \times \Delta)$-rectangles whose longer $\Delta$-sides have slope $\sigma$, with the following properties:
\begin{itemize}
\item[\textup{(1)} \phantomsection \label{1}] Let $\mathcal{T}_{\sigma}$ be the $\delta$-tubes spanned by the rectangles $\mathcal{R}_{\sigma}$. Every tube in $\mathcal{T}_{\sigma}$ contains exactly one element of $\mathcal{R}_{\sigma}$, and the tubes in $\mathcal{T}_{\sigma}$ are $10\delta$-separated.
\item[\textup{(2)} \phantomsection \label{2}] $\cup \mathfrak{P} \cap (\cup \{10T : T \in \mathcal{T}_{\sigma}\}) = \emptyset$.
\item[\textup{(3)} \phantomsection \label{3}] If $t \in [1 + \epsilon,\tfrac{3}{2}]$, and $\mathcal{Q} \subset \mathcal{D}_{\Delta}$  then 
\begin{displaymath} \mathcal{H}^{t - \epsilon}_{\infty}(\mathcal{Q} \cap \mathcal{D}_{\delta}(\cup \mathcal{R}_{\sigma})) \geq \mathcal{H}^{t}_{\infty}(\mathcal{Q}). \end{displaymath}
\end{itemize}
\end{proposition} 

\begin{remark} The family $\mathfrak{P}$ and the families $\mathcal{R}_{\sigma},\mathcal{T}_{\sigma}$ do not depend on $\mathcal{Q},t$. \end{remark} 

\begin{figure}[h!]
\begin{center}
\begin{overpic}[scale = 0.7]{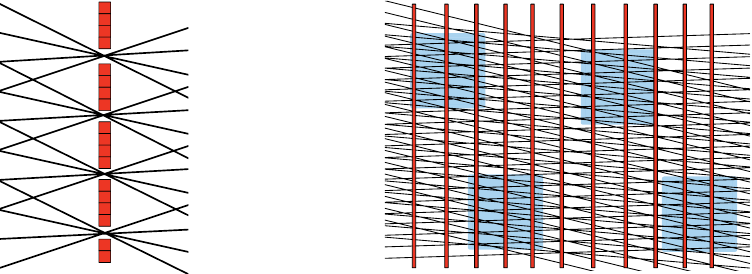}
\end{overpic}
\caption{The basic building block from Proposition \ref{prop1} at a microscopic scale (left) and a macroscopic scale (right). The picture on the left shows how the tubes $T \in \mathcal{T}_{\sigma}$ avoid $\mathfrak{P}$ at scale $\delta$. The picture on the right shows how both $\mathfrak{P}$ and $\mathcal{T}_{\sigma}$ have large intersection with any family of $\mathcal{Q} \in \mathcal{D}_{\Delta}$.}\label{fig2}
\end{center}
\end{figure}

From now on, the symbol $\theta$ will denote an interval of slopes, for example $\theta \in \mathcal{D}_{\delta}([0,1))$. The symbol $\sigma$ will refer to slopes (points in $[0,1)$).

\begin{definition}[$\theta$-disconnected pair]\label{def:disconnectedPair} Let $\delta \in 2^{-\N}$ and $\theta \in \mathcal{D}_{\delta}([0,1))$. A pair $\mathcal{P}_{1},\mathcal{P}_{2} \subset \mathcal{D}$ is \emph{$\theta$-disconnected} if there exists no line $\ell \subset \R^{2}$ with slope $\sigma(\ell) \in \theta$ such that 
\begin{displaymath} (\cup \mathcal{P}_{1}) \cap [\ell]_{\delta/2} \neq \emptyset \quad \text{and} \quad (\cup \mathcal{P}_{2}) \cap [\ell]_{\delta/2} \neq \emptyset. \end{displaymath} \end{definition}

We also need the following related notion:

\begin{definition}[$\theta$-diameter] Let $\delta \in 2^{-\N}$ and $\theta \in \mathcal{D}_{\delta}([0,1))$. The \emph{$\theta$-diameter of $\mathcal{P} \subset \mathcal{D}$} is defined by 
\begin{displaymath} \diam_{\theta}(\mathcal{P}) := \sup_{\sigma(\ell) \in \theta} \diam(\cup \mathcal{P} \cap [\ell]_{\delta/2}), \end{displaymath} 
where the "$\sup$" runs over all lines $\ell \subset \R^{2}$ with slope $\sigma(\ell) \in \theta$.\end{definition} 

\begin{thm}\label{thm1} For each $\epsilon,\Delta \in 2^{-\N}$ there exists $\delta = \delta(\epsilon,\Delta) \in 2^{-\N} \cap (0,\tfrac{1}{2}\Delta]$ such that the following holds.
For each $\boldsymbol{\theta} \subset \mathcal{D}_{\Delta}([0,1))$, let $\mathcal{P}_{\boldsymbol{\theta}} \subset \mathcal{D}_{\Delta}$ (an arbitrary family). 

Then, for each $\theta \in \mathcal{D}_{\delta}([0,1))$, there exists a family $\mathcal{P}_{\theta} \subset \mathcal{D}_{\delta}$ with the following properties:
\begin{itemize}
\item[\textup{(A)} \phantomsection \label{A}] $\mathcal{P}_{\theta} \subset \mathcal{D}_{\delta}(\mathcal{P}_{\boldsymbol{\theta}})$ for $\boldsymbol{\theta} \in \mathcal{D}_{\Delta}([0,1))$ and $\theta \in \mathcal{D}_{\delta}(\boldsymbol{\theta})$.
\item[\textup{(B)} \phantomsection \label{B}] $\diam_{\theta}(\mathcal{P}_{\theta}) \leq \Delta$ for all $\theta \in \mathcal{D}_{\delta}([0,1))$.
\item[\textup{(C)} \phantomsection \label{C}] $\mathcal{H}^{t - \epsilon}_{\infty}(\mathcal{P}_{\theta}) \geq \mathcal{H}^{t}_{\infty}(\mathcal{P}_{\boldsymbol{\theta}})$ for all $\boldsymbol{\theta} \in \mathcal{D}_{\Delta}([0,1))$, $\theta \in \mathcal{D}_{\delta}(\boldsymbol{\theta})$, and $t \in [1 + \epsilon,\tfrac{3}{2}]$.
\item[\textup{(D)} \phantomsection \label{D}] If $\boldsymbol{\theta},\boldsymbol{\theta}' \in \mathcal{D}_{\Delta}([0,1))$ are distinct, and $\theta \in \mathcal{D}_{\delta}(\boldsymbol{\theta})$ and $\theta' \in \mathcal{D}_{\delta}(\boldsymbol{\theta}')$, then the pair $\mathcal{P}_{\theta},\mathcal{P}_{\theta'}$ is $\theta$-disconnected and $\theta'$-disconnected.
\end{itemize}
\end{thm} 

Theorem \ref{thm1} will be proved by iterating the following auxiliary result:

\begin{lemma}\label{prop2} For every $\epsilon,\Delta \in 2^{-\N}$ and $M \in \N$ there exists $\delta = \delta(\Delta,\epsilon,M) \in (0,\tfrac{1}{2}\Delta]$ such that the following holds. Assume that for each $\boldsymbol{\theta} \in \mathcal{D}_{\Delta}([0,1))$ we are given a family $\mathcal{P}_{\boldsymbol{\theta}} \subset \mathcal{D}_{\Delta}$. Let $\mathcal{M}$ be a finite family of subsets of $\mathcal{D}_{\Delta}$ with $|\mathcal{M}| \leq M$.

First, for each $\theta \in \mathcal{D}_{\delta}([0,1))$ there exists a family $\mathcal{P}_{\theta} \subset \mathcal{D}_{\delta}$ with the following properties:
\begin{itemize}
\item[\textup{(i)} \phantomsection \label{i}] If $\boldsymbol{\theta} \in \mathcal{D}_{\Delta}([0,1))$, and $\theta \in \mathcal{D}_{\delta}(\boldsymbol{\theta})$, then $\mathcal{P}_{\theta} \subset \mathcal{D}_{\delta}(\mathcal{P}_{\boldsymbol{\theta}})$. Moreover $\diam_{\theta}(\mathcal{P}_{\theta}) \leq \Delta$.
\item[\textup{(ii)} \phantomsection \label{ii}] $\mathcal{H}^{t - \epsilon}_{\infty}(\mathcal{P}_{\theta}) \geq \mathcal{H}^{t}(\mathcal{P}_{\boldsymbol{\theta}})$ for all $\theta \in \mathcal{D}_{\delta}(\boldsymbol{\theta})$ and $\boldsymbol{\theta} \in \mathcal{D}_{\Delta}([0,1))$, and all $t \in [1 + \epsilon,\tfrac{3}{2}]$.
\end{itemize}
Second, for each $\mathcal{F} \in \mathcal{M}$ there exists a subset $\mathcal{F}' \subset \mathcal{D}_{\delta}(\mathcal{F})$ such that $\mathcal{H}^{t - \epsilon}_{\infty}(\mathcal{F}') \geq \mathcal{H}^{t}_{\infty}(\mathcal{F})$ for all $t \in [1 + \epsilon,\tfrac{3}{2}]$. The family of these subsets $\mathcal{F}'$ is denoted $\mathcal{M}'$.

Third, every pair $\mathcal{P}_{\theta},\mathcal{F}'$ is $\theta$-disconnected, with $\theta \in \mathcal{D}_{\delta}([0,1))$ and $\mathcal{F}' \in \mathcal{M}'$. \end{lemma} 

\begin{proof} We start by reducing the proof to the case where $M = 1$. Indeed, if $M > 1$, we apply the proposition $M$ times consecutively. We give the details for the case $M = 2$, thus $\mathcal{M} = \{\mathcal{F}_{1},\mathcal{F}_{2}\}$; the general case should be clear after this, and only notationally messier.

Start by applying the case $M = 1$ of the proposition with constants $\Delta,\epsilon/2$ and data 
\begin{displaymath} \{\mathcal{P}_{\boldsymbol{\theta}}\}_{\boldsymbol{\theta} \in \mathcal{D}_{\Delta}([0,1))} \quad \text{and} \quad \{\mathcal{F}_{1}\}. \end{displaymath}
This yields 
\begin{itemize}
\item a scale $\delta_{0} \in 2^{-\N} \cap (0,\tfrac{1}{2}\Delta]$, 
\item families $\mathcal{P}_{\theta} \subset \mathcal{D}_{\delta_{0}}$, $\theta \in \mathcal{D}_{\delta_{0}}([0,1))$, which satisfy \nref{i}-\nref{ii} (with constant $\epsilon/2$), and
\item a family $\mathcal{F}_{1}' \subset \mathcal{D}_{\delta_{0}}(\mathcal{F}_{1})$, such that $\mathcal{H}^{t - \epsilon}_{\infty}(\mathcal{F}_{1}') \geq \mathcal{H}^{t}_{\infty}(\mathcal{F}_{1})$ for every $t \in [1 + \epsilon,\tfrac{3}{2}]$, and every pair $\mathcal{P}_{\theta},\mathcal{F}_{1}'$ is $\theta$-disconnected, $\theta \in \mathcal{D}_{\delta_{0}}([0,1))$.
\end{itemize}

Next, we reapply the case $M = 1$ of the proposition, with constants $\delta_{0},\epsilon/2$ and data 
\begin{displaymath} \{\mathcal{P}_{\theta}\}_{\theta \in \mathcal{D}_{\delta_{0}}([0,1))} \quad \text{and} \quad \{\mathcal{F}_{2}\}. \end{displaymath}
This gives another scale $\delta \in 2^{-\N} \cap (0,\delta_{0}]$, new families $\mathcal{P}_{\theta'} \subset \mathcal{D}_{\delta}$, $\theta' \in \mathcal{D}_{\delta}([0,1))$, satisfying \nref{i}-\nref{ii} relative to the families $\{\mathcal{P}_{\theta}\}_{\theta \in \mathcal{D}_{\delta_{0}}([0,1))}$, and finally a set $\mathcal{F}_{2}' \in \mathcal{D}_{\delta}(\mathcal{F}_{2})$, such that every pair $\mathcal{P}_{\theta'},\mathcal{F}_{2}'$ is $\theta'$-disconnected, for $\theta' \in \mathcal{D}_{\delta}([0,1))$. 

We claim that the collections $\{\mathcal{P}_{\theta'}\}_{\theta' \in \mathcal{D}_{\delta}([0,1))}$ and $\mathcal{M}' = \{\mathcal{F}_{1}',\mathcal{F}_{2}'\}$ satisfy all the claims of Lemma \ref{prop2}. To be precise, we redefine $\mathcal{F}_{1}' := \mathcal{D}_{\delta}(\mathcal{F}_{1}')$ to comply with the requirement $\mathcal{F}_{1}',\mathcal{F}_{2}' \subset \mathcal{D}_{\delta}$. This has no effect on the properties claimed about $\mathcal{F}_{1}'$ so far (Hausdorff content, disconnectedness of $\mathcal{P}_{\theta},\mathcal{F}_{1}'$).

We start with \nref{i}-\nref{ii}. Fix $\boldsymbol{\theta} \in \mathcal{D}_{\Delta}([0,1))$ and $\theta' \in \mathcal{D}_{\delta}(\boldsymbol{\theta})$. Let $\theta \in \mathcal{D}_{\delta_{0}}([0,1))$ be the unique dyadic arc with $\theta' \subset \theta \subset \boldsymbol{\theta}$. Then $\mathcal{P}_{\theta'} \subset \mathcal{P}_{\theta} \subset \mathcal{P}_{\boldsymbol{\theta}}$ by construction, and also
\begin{displaymath} \mathcal{H}_{\infty}^{t - \epsilon}(\mathcal{P}_{\theta'}) \geq \mathcal{H}_{\infty}^{t - \epsilon/2}(\mathcal{P}_{\theta}) \geq \mathcal{H}_{\infty}^{t}(\mathcal{P}_{\boldsymbol{\theta}}), \qquad t \in [1 + \epsilon,\tfrac{3}{2}]. \end{displaymath} 
This verifies \nref{i}-\nref{ii} (note also that $\diam_{\theta'}(\mathcal{P}_{\theta'}) \leq \delta_{0} \leq \Delta$ for all $\theta' \in \mathcal{D}_{\delta}([0,1))$.) The families $\mathcal{F}_{j}'$, $j \in \{1,2\}$, evidently satisfy $\mathcal{H}^{t - \epsilon}_{\infty}(\mathcal{F}_{j}') \geq \mathcal{H}^{t}_{\infty}(\mathcal{F}_{j})$ for each $t \in [1 + \epsilon,\tfrac{3}{2}]$. So, we only need to check that every pair $\mathcal{P}_{\theta'},\mathcal{F}_{j}'$ is $\theta'$-disconnected, for $\theta' \in \mathcal{D}_{\delta}([0,1))$ and $j \in \{1,2\}$. This is clear for $j = 2$ by the choice of the sets $\mathcal{P}_{\theta'}$ and $\mathcal{F}_{2}'$. So, it remains to check that every pair $\mathcal{P}_{\theta}',\mathcal{F}_{1}'$ is $\theta$-disconnected, for $\theta' \in \mathcal{D}_{\delta}([0,1))$.

For this we use property \nref{i}. Fix $\theta' \in \mathcal{D}_{\delta}([0,1))$, and let $\theta \in \mathcal{D}_{\delta_{0}}([0,1))$ be the dyadic parent of $\theta'$. Then $\mathcal{P}_{\theta'} \subset \mathcal{D}_{\delta}(\mathcal{P}_{\theta})$ according to \nref{i}, and moreover $\mathcal{P}_{\theta},\mathcal{F}_{1}'$ is $\theta$-disconnected by the first step of the construction. This means that if $\ell \subset \R^{2}$ is a line with $\sigma(\ell) \in \theta$, then $[\ell]_{\delta_{0}}$ cannot intersect both $\cup \mathcal{P}_{\theta}$ and $\cup \mathcal{F}_{1}'$. This certainly implies that if $\ell \subset \R^{2}$ is a line with $\sigma(\ell) \in \theta' \subset \theta$, then $[\ell]_{\delta/2}$ cannot intersect both $\cup \mathcal{P}_{\theta'} \subset \cup \mathcal{P}_{\theta}$ and $\cup \mathcal{F}_{1}'$. Therefore $\mathcal{P}_{\theta'},\mathcal{F}_{1}'$ is $\theta'$-disconnected. This concludes the proof of the case $M = 2$ (the proofs in the cases $M \geq 3$ are similar and left to the reader).

It remains to prove the case $M = 1$. Write $\mathcal{M} =: \{\mathcal{F}\} \subset \mathcal{D}_{\Delta}$. Recall that the families $\mathcal{P}_{\boldsymbol{\theta}} \subset \mathcal{D}_{\Delta}$ are indexed by the arcs $\boldsymbol{\theta} \in \mathcal{D}_{\Delta}([0,1))$. Let us enumerate these arcs $\{\boldsymbol{\theta}_{1},\ldots,\boldsymbol{\theta}_{m}\}$, where $m = \Delta^{-1}$. We first consider the pair $(\mathcal{P}_{\boldsymbol{\theta}_{1}},\mathcal{F})$. 

We apply Proposition \ref{prop1} with parameters $\epsilon/m = \Delta \epsilon$ and $\tfrac{1}{2}\Delta$. This gives a scale $\delta_{1} \in 2^{-\N} \cap (0,\tfrac{1}{2}\Delta]$, and a family $\mathfrak{P}_{1} \subset \mathcal{D}_{\delta_{1}}$ with satisfying $\mathcal{H}_{\infty}^{t - \epsilon/m}(\mathcal{F} \cap \mathfrak{P}_{1}) \geq \mathcal{H}^{t}_{\infty}(\mathcal{F})$ for all $t \in [\epsilon/m,\tfrac{3}{2}]$. Set $\mathcal{F}_{1} := \mathcal{F} \cap \mathfrak{P}_{1}$, thus 
\begin{equation}\label{form13} \mathcal{F}_{1} \subset \mathcal{D}_{\delta_{1}}(\mathcal{F}) \quad \text{and} \quad \mathcal{H}^{t - \epsilon/m}_{\infty}(\mathcal{F}_{1}) \geq \mathcal{H}^{t}_{\infty}(\mathcal{F}) \text{ for } t \in [1 + \epsilon,\tfrac{3}{2}]. \end{equation}
Moreover, for each $\theta \in \mathcal{D}_{\delta_{1}}([0,1))$ (in particular $\theta \in \mathcal{D}_{\delta_{1}}(\boldsymbol{\theta}_{1})$), Proposition \ref{prop1} gives a family $\mathcal{T}_{\theta}$ of $\delta_{1}$-tubes parallel to (the slope determined by the left endpoint of) $\theta$, and associated $(\delta \times \Delta/2)$-rectangles $\mathcal{R}_{\theta}$, such that
\begin{itemize}
\item[(a) \phantomsection \label{a}] $\cup \mathcal{F}_{1} \cap (\cup \{10T : T \in \mathcal{T}_{\theta}\}) = \emptyset$,
\item[(b) \phantomsection \label{b}] $\mathcal{H}^{t - \epsilon}_{\infty}(\mathcal{P}_{\boldsymbol{\theta}_{1}} \cap \mathcal{D}_{\delta_{1}}(\cup \mathcal{R}_{\theta})) \geq \mathcal{H}^{t}_{\infty}(\mathcal{P}_{\boldsymbol{\theta}_{1}})$, for $\theta \in \mathcal{D}_{\delta_{1}}(\boldsymbol{\theta}_{1})$ and $t \in [1 + \epsilon,\tfrac{3}{2}]$.
\end{itemize}
We now define 
\begin{displaymath} \mathcal{P}_{\theta} := \mathcal{P}_{\boldsymbol{\theta}_{1}} \cap \mathcal{D}_{\delta_{1}}(\cup \mathcal{R}_{\theta}), \qquad \theta \in \mathcal{D}_{\delta_{1}}(\boldsymbol{\theta}_{1}). \end{displaymath}
Thus $\mathcal{H}^{t - \epsilon}_{\infty}(\mathcal{P}_{\theta}) \geq \mathcal{H}^{t}_{\infty}(\mathcal{P}_{\boldsymbol{\theta}_{1}})$ for all $t \in [1 + \epsilon,\tfrac{3}{2}]$, and $\mathcal{P}_{\theta} \subset \mathcal{D}_{\delta_{1}}(\mathcal{P}_{\boldsymbol{\theta}_{1}})$ for $\theta \in \mathcal{D}_{\delta_{1}}(\boldsymbol{\theta}_{1})$. This means that the sets $\mathcal{P}_{\theta}$ (for $\theta \in \mathcal{D}_{\delta_{1}}(\boldsymbol{\theta}_{1})$) satisfy the requirements \nref{i}-\nref{ii} in the statement of Lemma \ref{prop2}. The claim $\diam_{\theta}(\mathcal{P}_{\theta}) \leq \Delta$ follows from the inclusion $\mathcal{P}_{\theta} \subset \mathcal{D}_{\delta_{1}}(\cup \mathcal{R}_{\theta})$ (and $\delta_{1} \leq \Delta/2$), and recalling the key properties of the rectangles and tubes $\mathcal{R}_{\theta},\mathcal{T}_{\theta}$ from Proposition \ref{prop1}: each tube $T \in \mathcal{T}_{\theta}$ only contains one rectangle from $\mathcal{R}_{\theta}$, the tubes $T \in \mathcal{T}_{\theta}$ have slope $\theta$, and the tubes $T \in \mathcal{T}_{\theta}$ are $10\delta$-separated.

A small caveat is that "$\delta_{1}$" is not yet the final scale $\delta \in 2^{-\N} \cap (0,\delta_{0}]$ announced in the proposition, and we claimed in the statement to choose $\mathcal{P}_{\theta} \subset \mathcal{D}_{\delta}$. Once the final scale $\delta$ has been located (after altogether $m$ steps of the argument), we will re-define the sets $\mathcal{P}_{\theta}$ as follows: for each $\theta \in \mathcal{D}_{\delta_{1}}(\boldsymbol{\theta}_{1})$ and $\theta' \in \mathcal{D}_{\delta}(\theta)$, we set $\mathcal{P}_{\theta'} := \mathcal{D}_{\delta}(\mathcal{P}_{\theta})$. Then $\mathcal{H}^{t - \epsilon}_{\infty}(\mathcal{P}_{\theta'}) \geq \mathcal{H}^{t}_{\infty}(\mathcal{P}_{\boldsymbol{\theta}_{1}})$ for all $\theta' \in \mathcal{D}_{\delta}(\boldsymbol{\theta}_{1})$. We will next check that the pairs $\mathcal{P}_{\theta},\mathcal{F}_{1}$ are $\theta$-disconnected. This will clearly imply that the pairs $\mathcal{P}_{\theta'},\mathcal{F}_{1}$ are $\theta'$-disconnected. 

Let us check that $\mathcal{P}_{\theta},\mathcal{F}_{1}$ is $\theta$-disconnected for all $\theta \in \mathcal{D}_{\delta_{1}}(\boldsymbol{\theta}_{1})$. Indeed, fix $\theta \in \mathcal{D}_{\delta_{1}}(\boldsymbol{\theta}_{1})$, and let $\ell \subset \R^{2}$ be a line with $\sigma(\ell) \in \theta$. Now, if $[\ell]_{\delta_{1}/2}$ intersects $\cup \mathcal{P}_{\theta} \subset \cup \mathcal{D}_{\delta_{1}}(\cup \mathcal{T}_{\theta})$, then $[\ell]_{\delta_{1}/2} \subset \cup \{10T : T \in \mathcal{T}_{\theta}\}$. Therefore $[\ell]_{\delta_{1}/2}$ does not intersect $\cup \mathcal{F}_{1}$ by \nref{a}.

Next we deal with the arc $\boldsymbol{\theta}_{k}$ with $2 \leq k \leq m$. Assume that we have already constructed a decreasing sequence of scales $\delta_{1},\ldots,\delta_{k - 1} \in 2^{-\N} \cap (0,\Delta]$, and a nested sequence of families $\mathcal{F}_{i} \subset \mathcal{D}_{\delta_{i}}(\mathcal{F})$, $1 \leq i \leq k - 1$, satisfying 
\begin{equation}\label{form14} \mathcal{H}_{\infty}^{t - i\epsilon/m}(\mathcal{F}_{i}) \geq \mathcal{H}_{\infty}^{t}(\mathcal{F}), \qquad 1 \leq i \leq k - 1, \, t \in [1 + \epsilon,\tfrac{3}{2}]. \end{equation}
(By "nested" we mean that the squares in $\mathcal{F}_{i + 1}$ are contained in the union of the squares in $\mathcal{F}_{i}$.) The case $k = 2$ is true by \eqref{form13}. We also assume that for each $1 \leq i \leq k - 1$, and for each $\theta \in \mathcal{D}_{\delta_{i}}(\boldsymbol{\theta}_{i})$, we have constructed a family $\mathcal{P}_{\theta} \subset \mathcal{D}_{\delta_{i}}(\mathcal{P}_{\boldsymbol{\theta}_{i}})$ such that $\mathcal{P}_{\theta},\mathcal{F}_{k - 1}$ is $\theta$-disconnected. We verified this just above in the case $k = 2$.

We next construct the scale $\delta_{k} \in 2^{-\N} \cap (0,\delta_{k - 1}]$, the family $\mathcal{F}_{k} \subset \mathcal{D}_{\delta_{k}}(\mathcal{F}_{k - 1}) \subset \mathcal{D}_{\delta_{k}}(\mathcal{F})$, and the sets $\mathcal{P}_{\theta} \subset \mathcal{D}_{\delta_{k}}(\mathcal{P}_{\boldsymbol{\theta}_{k}})$ for $\theta \in \mathcal{D}_{\delta_{k}}(\boldsymbol{\theta}_{k})$. To do this, we apply Proposition \ref{prop1} with parameters $s = 3/2$, $\epsilon/m = \Delta \epsilon$, and scale $\delta_{k - 1}$. This gives a scale $\delta_{k} \in 2^{-\N} \cap (0,\delta_{1}]$, and a family $\mathfrak{P}_{k} \subset \mathcal{D}_{\delta_{k}}$ with the property 
\begin{equation}\label{form15} \mathcal{H}^{t - k \epsilon/m}(\mathcal{F}_{k - 1} \cap \mathfrak{P}_{k}) \geq \mathcal{H}^{t - (k - 1)\epsilon/m}(\mathcal{F}_{k - 1}) \stackrel{\eqref{form14}}{\geq} \mathcal{H}^{t}_{\infty}(\mathcal{F}), \quad t \in [1 + \epsilon,\tfrac{3}{2}]. \end{equation}
Define $\mathcal{F}_{k} := \mathcal{F}_{k - 1} \cap \mathfrak{P}_{k}$. Moreover, for each $\theta \in \mathcal{D}_{\delta_{k}}(\boldsymbol{\theta}_{k})$, Proposition \ref{prop1} gives a family $\mathcal{T}_{\theta}$ of $\delta_{k}$-tubes parallel to $\theta$, and associated $(\delta_{k} \times \delta_{k - 1})$-rectangles $\mathcal{R}_{\theta}$, such that
\begin{itemize}
\item[($a_{k}$)] $\cup \mathcal{F}_{k} \cap (\cup \{10T : T \in \mathcal{T}_{\theta}\}) = \emptyset$, and
\item[($b_{k}$)] $\mathcal{H}^{t - \epsilon}_{\infty}(\mathcal{P}_{\boldsymbol{\theta}_{k}} \cap \mathcal{D}_{\delta_{k}}(\cup \mathcal{R}_{\theta})) \geq \mathcal{H}^{t}_{\infty}(\mathcal{P}_{\boldsymbol{\theta}_{k}})$, for $\theta \in \mathcal{D}_{\delta_{k}}(\boldsymbol{\theta}_{k})$ and $t \in [1 + \epsilon,\tfrac{3}{2}]$.
\end{itemize}
We now define 
\begin{displaymath} \mathcal{P}_{\theta} := \mathcal{P}_{\boldsymbol{\theta}_{k}} \cap \mathcal{D}_{\delta_{k}}(\cup \mathcal{R}_{\theta}), \qquad \theta \in \mathcal{D}_{\delta_{k}}(\boldsymbol{\theta}_{k}). \end{displaymath}
Thus $\mathcal{P}_{\theta} \subset \mathcal{D}_{\delta_{k}}(\mathcal{P}_{\boldsymbol{\theta}_{k}})$ and $\mathcal{H}^{t - \epsilon}(\mathcal{P}_{\theta}) \geq \mathcal{H}^{t}_{\infty}(\mathcal{P}_{\boldsymbol{\theta}_{k}})$ for all $\theta \in \mathcal{D}_{\delta_{k}}(\boldsymbol{\theta}_{k})$ and $t \in [1 + \epsilon,\tfrac{3}{2}]$. We also note that $\diam_{\theta}(\mathcal{P}_{\theta}) \leq 2\delta_{k - 1} \leq \Delta$ for all $\theta \in \mathcal{D}_{\delta_{k}}(\boldsymbol{\theta}_{k})$.

To complete the induction, it remains to check that every pair $\mathcal{P}_{\theta},\mathcal{F}_{k}$ is $\theta$-disconnected for $\theta \in \mathcal{D}_{\delta_{i}}(\boldsymbol{\theta}_{i})$, $1 \leq i \leq k$. For $1 \leq i \leq k - 1$, this follows immediately from the $\theta$-disconnectedness of $\mathcal{P}_{\theta},\mathcal{F}_{k - 1}$, and $\cup \mathcal{F}_{k} \subset \cup \mathcal{F}_{k - 1}$. For $i = k$, the argument is the same as the argument we already recorded above for the $\theta$-disconnectedness of the pairs $\mathcal{P}_{\theta},\mathcal{F}_{1}$, with $\theta \in \mathcal{D}_{\delta_{1}}(\boldsymbol{\theta}_{1})$, now using ($a_{k}$) and $\mathcal{P}_{\theta} \subset \mathcal{D}_{\delta_{k}}(\cup \mathcal{T}_{\theta})$. We do not repeat the details.

We complete the proof by setting $\delta := \delta_{m}$, and $\mathcal{F}' := \mathcal{F}_{m}$. Then $\mathcal{H}^{t - \epsilon}(\mathcal{F}') \geq \mathcal{H}^{t}_{\infty}(\mathcal{F})$ for all $t \in [1 + \epsilon,\tfrac{3}{2}]$ by \eqref{form15} (with $k = m$). As already mentioned when dealing with the case $k = 1$, we finally have to redefine all the sets $\mathcal{P}_{\theta} \subset \mathcal{D}_{\delta_{k}}(\boldsymbol{\theta}_{k})$, $1 \leq k \leq m$, in a trivial way, to comply with the requirement that they are all subsets of $\mathcal{D}_{\delta}$, and indexed by $\mathcal{D}_{\delta}([0,1))$. 

For $1 \leq k \leq m$, $\theta \in \mathcal{D}_{\delta_{k}}(\boldsymbol{\theta}_{k})$, and $\theta' \in \mathcal{D}_{\delta}(\theta)$, we set $\mathcal{P}_{\theta'} := \mathcal{D}_{\delta}(\mathcal{P}_{\theta})$. With this definition $\mathcal{P}_{\theta'} \subset \mathcal{D}_{\delta}(\mathcal{P}_{\boldsymbol{\theta}})$ and $\mathcal{H}^{t - \epsilon}_{\infty}(\mathcal{P}_{\theta'}) \geq \mathcal{H}^{t}_{\infty}(\mathcal{P}_{\boldsymbol{\theta}})$ for all $\boldsymbol{\theta} \in \mathcal{D}_{\Delta}([0,1))$, all $\theta' \in \mathcal{D}_{\delta}([0,1))$, and all $t \in [1 + \epsilon,\tfrac{3}{2}]$. The $\theta'$-disconnectedness of the pairs $\mathcal{P}_{\theta'},\mathcal{F}'$ follows from the $\theta$-disconnectedness of the pairs $\mathcal{P}_{\theta},\mathcal{F}'$. The bound $\diam_{\theta'}(\mathcal{P}_{\theta'}) \leq \Delta$ is inherited from $\diam_{\theta}(\mathcal{P}_{\theta}) \leq \Delta$. \end{proof} 
 
 We are then equipped to prove Theorem \ref{thm1}.
 
\begin{proof}[Proof of Theorem \ref{thm1}] Recall that $\Delta \in 2^{-\N}$. We may assume that $\Delta \in (0,\tfrac{1}{2}]$. Namely, if $\Delta = 1$ then we are given a single family $\mathcal{P}_{[0,1)} \subset \mathcal{D}_{1}$, thus $\mathcal{P}_{[0,1)} = \{[0,1)^{2}\}$. We then define $\delta := 1/2$ and $\mathcal{P}_{[0,1/2)} := \mathcal{D}_{1/2} =: \mathcal{P}_{[1/2,1)}$. The reader may check that all properties \nref{A}-\nref{D} are satisfies (\nref{D} is vacuous, because $\mathcal{D}_{1}([0,1))$ contains a single interval).

We then assume that $\Delta \in (0,\tfrac{1}{2}]$. Let us enumerate 
\begin{displaymath} \mathcal{D}_{\Delta}([0,1)) =: \{\boldsymbol{\theta}_{1},\ldots,\boldsymbol{\theta}_{m}\} \end{displaymath}
with $m = \Delta^{-1} \geq 2$. The proof consists of $m$ applications of Lemma \ref{prop2}. At step $k \in \{1,\ldots,m\}$, "the sets associated to $\boldsymbol{\theta}_{k}$ are disconnected from the sets associated to all the other arcs $\boldsymbol{\theta}_{j}$ with $j \neq k$." Naturally the sentence above does not make full sense at the moment, since the "sets associated to $\boldsymbol{\theta}_{j}$" are not properly defined. 

Write $\delta_{0} := \Delta$. Fix $0 \leq k < m$, and assume that the following objects have been constructed. First, a decreasing sequence of scales $\delta_{0},\delta_{1},\ldots,\delta_{k} \in 2^{-\N} \cap (0,\Delta]$. 

Second, for each $0 \leq j \leq k$ and $\theta \in \mathcal{D}_{\delta_{j}}([0,1))$, a family $\mathcal{P}_{\theta} \subset \mathcal{D}_{\delta_{j}}$ as follows:
\begin{itemize}
\item[(I1) \phantomsection \label{I1}] If $0 \leq i \leq j$, $\theta_{i} \in \mathcal{D}_{\delta_{i}}([0,1))$ and $\theta_{j} \in \mathcal{D}_{\delta_{j}}(\theta_{i})$, then $\mathcal{P}_{\theta_{j}} \subset \mathcal{D}_{\delta_{j}}(\mathcal{P}_{\theta_{i}})$. 
\item[(I2) \phantomsection \label{I2}] If $\theta \in \mathcal{D}_{\delta_{k}}(\boldsymbol{\theta})$ for some $\boldsymbol{\theta} \in \mathcal{D}_{\Delta}([0,1))$, then
\begin{displaymath} \mathcal{H}^{t - k\epsilon/m}_{\infty}(\mathcal{P}_{\theta}) \geq \mathcal{H}^{t}_{\infty}(\mathcal{P}_{\boldsymbol{\theta}}), \qquad t \in [1 + \epsilon,\tfrac{3}{2}]. \end{displaymath}
\item[(I3) \phantomsection \label{I3}] If $1 \leq i \leq m$ and $j \in \{1,\ldots,k\} \, \setminus \, \{i\}$, then the following holds. Let $\theta_{i} \in \mathcal{D}_{\delta_{k}}(\boldsymbol{\theta}_{i})$ and $\theta_{j} \in \mathcal{D}_{\delta_{k}}(\boldsymbol{\theta}_{j})$. Then the pair $\mathcal{P}_{\theta_{i}},\mathcal{P}_{\theta_{j}}$ is $\theta_{i}$-disconnected.
\item[(I4)\phantomsection \label{I4}] If $1 \leq k \leq m$ and $j \in \{1,\ldots,m\} \, \setminus \, \{k\}$, then $\diam_{\theta}(\mathcal{P}_{\theta}) \leq \Delta$ for all $\theta \in \mathcal{D}_{\delta_{k}}(\boldsymbol{\theta}_{j})$.
\end{itemize}
Note that conditions \nref{I1}-\nref{I2} with $k = 0$ are trivially satisfied by the families $\mathcal{P}_{\boldsymbol{\theta}_{1}},\ldots,\mathcal{P}_{\boldsymbol{\theta}_{m}}$ provided to us by the statement of the proposition, while \nref{I3}-\nref{I4} are vacuous for $k = 0$.

We will next construct the scale $\delta_{k + 1}$ and the families $\mathcal{P}_{\theta} \subset \mathcal{D}_{\delta_{k + 1}}$, $\theta \in \mathcal{D}_{\delta_{k + 1}}([0,1))$, by applying Lemma \ref{prop2} with parameters $\delta_{k} \in 2^{-\N} \cap (0,\Delta]$, $\epsilon/m$, and $M \in \N$ to be determined in a moment. The families $\mathcal{P}_{\theta} \subset \mathcal{D}_{\delta_{k}}$, $\theta \in \mathcal{D}_{\delta_{k}}([0,1))$, to which Lemma \ref{prop2} is applied are the ones provided by the inductive hypothesis. The family $\mathcal{M}$ is defined by
\begin{displaymath} \mathcal{M} := \{\mathcal{P}_{\theta_{k}} : \theta_{k} \in \mathcal{D}_{\delta_{k}}(\boldsymbol{\theta}_{k + 1})\}. \end{displaymath}
Set $M := |\mathcal{M}|$. The results of the application of Lemma \ref{prop2} are the following. We are provided with a scale $\delta_{k + 1} = \delta_{k + 1}(\delta_{k},\epsilon/m,M) \in (0,\tfrac{1}{2}\delta_{k}]$. For each $\theta \in \mathcal{D}_{\delta_{k + 1}}([0,1))$ we are provided with a family $\overline{\mathcal{P}}_{\theta} \subset \mathcal{D}_{\delta_{k + 1}}$ with the following properties (we use the overline symbol to signify that the sets $\overline{\mathcal{P}}_{\theta}$ are not the "final" sets of generation $k + 1$):
\begin{itemize}
\item[(a)] If $\theta_{k} \in \mathcal{D}_{\delta_{k}}([0,1))$ and $\theta \in \mathcal{D}_{\delta_{k + 1}}(\theta_{k})$, then $\overline{\mathcal{P}}_{\theta} \subset \mathcal{D}_{\delta_{k + 1}}(\mathcal{P}_{\theta_{k}})$. 
\item[(b)] $\diam_{\theta}(\overline{\mathcal{P}}_{\theta}) \leq \delta_{k} \leq \Delta$ for all $\theta \in \mathcal{D}_{\delta_{k + 1}}([0,1))$.
\item[(c)] If $\theta_{k} \in \mathcal{D}_{\delta_{k}}([0,1))$ and $\theta \in \mathcal{D}_{\delta_{k + 1}}(\theta_{k})$, then
\begin{displaymath} \mathcal{H}^{t - (k + 1)\epsilon/m}_{\infty}(\overline{\mathcal{P}}_{\theta}) \geq \mathcal{H}^{t - k\epsilon/m}(\mathcal{P}_{\theta_{k}}), \qquad t \in [1 + \epsilon,\tfrac{3}{2}]. \end{displaymath}
\end{itemize}
Moreover, for each $\mathcal{P}_{\theta_{k}} \in \mathcal{M}$, $\theta_{k} \in \mathcal{D}_{\delta_{k}}(\boldsymbol{\theta}_{k + 1})$, we are provided with a family $\mathcal{P}_{\theta_{k}}' \subset \mathcal{D}_{\delta_{k + 1}}(\mathcal{P}_{\theta_{k}})$ such that 
\begin{displaymath} \mathcal{H}^{t - (k + 1)\epsilon/m}_{\infty}(\mathcal{P}_{\theta_{k}}') \geq \mathcal{H}^{t - k\epsilon/m}_{\infty}(\mathcal{P}_{\theta_{k}}), \qquad t \in [1 + \epsilon,\tfrac{3}{2}], \end{displaymath}
and such that 
\begin{equation}\label{form18} \text{every pair $\overline{\mathcal{P}}_{\theta},\mathcal{P}_{\theta_{k}}'$ is $\theta$-disconnected, for $\theta \in \mathcal{D}_{\delta_{k + 1}}([0,1))$ and $\theta_{k} \in \mathcal{D}_{\delta_{k}}(\boldsymbol{\theta}_{k + 1})$.} \end{equation}

We are now prepared to define the sets $\mathcal{P}_{\theta}$, $\theta \in \mathcal{D}_{\delta_{k + 1}}([0,1))$, of generation $k + 1$. Set
\begin{equation}\label{form17} \mathcal{P}_{\theta} := \begin{cases} \overline{\mathcal{P}}_{\theta}, & \theta \in \mathcal{D}_{\delta_{k + 1}}([0,1)) \, \setminus \, \mathcal{D}_{\delta_{k + 1}}(\boldsymbol{\theta}_{k + 1}), \\ \mathcal{P}_{\theta_{k}}', & \theta \in \mathcal{D}_{\delta_{k + 1}}(\theta_{k}), \, \theta_{k} \in \mathcal{D}_{\delta_{k}}(\boldsymbol{\theta}_{k + 1}). \end{cases} \end{equation} 

We then verify that these sets $\mathcal{P}_{\theta}$ satisfy \nref{I1}-\nref{I4}. To verify \nref{I1}, fix $0 \leq i \leq j \leq k + 1$. If $j \leq k$, the property \nref{I1} follows from induction, and the case $i = k + 1 = j$ is trivial, so we only need to consider pairs $(i,k + 1)$ with $i < k$. If $\theta_{i} \in \mathcal{D}_{\delta_{i}}([0,1))$ and $\theta \in \mathcal{D}_{\delta_{k + 1}}(\theta_{i})$, we need to prove that $\mathcal{P}_{\theta} \subset \mathcal{D}_{\delta_{k + 1}}(\mathcal{P}_{\theta_{i}})$. In both cases of the definition of $\mathcal{P}_{\theta}$, the following is true: if $\theta_{k} \in \mathcal{D}_{\delta_{k}}([0,1))$ is the dyadic parent of $\theta$, then $\mathcal{P}_{\theta} \subset \mathcal{D}_{\delta_{k + 1}}(\mathcal{P}_{\theta_{k}})$. Now the claim follows by applying the inductive hypothesis to the pair $(i,k)$.

Next, we verify property \nref{I2}. Fix $\theta \in \mathcal{D}_{\delta_{k + 1}}(\boldsymbol{\theta})$ with $\boldsymbol{\theta} \in \mathcal{D}_{\Delta}([0,1))$. Let $\theta_{k} \in \mathcal{D}_{\delta_{k}}(\boldsymbol{\theta})$ be the dyadic parent of $\theta$. In both cases of the definition of $\mathcal{P}_{\theta}$, it holds 
\begin{displaymath} \mathcal{H}^{t - (k + 1)\epsilon/m}_{\infty}(\mathcal{P}_{\theta}) \geq \mathcal{H}_{\infty}^{t - k\epsilon/m}(\mathcal{P}_{\theta_{k}}) \geq \mathcal{H}^{t}_{\infty}(\mathcal{P}_{\boldsymbol{\theta}}), \qquad t \in [1 + \epsilon,\tfrac{3}{2}], \end{displaymath}
where the second inequality uses \nref{I2} inductively. This verifies \nref{I2}.

We move on to property \nref{I3}. Fix $1 \leq i \leq m$ and $j \in \{1,\ldots,k + 1\} \, \setminus \, \{i\}$, and let $\theta_{i} \in \mathcal{D}_{\delta_{k + 1}}(\boldsymbol{\theta}_{i})$ and $\theta_{j} \in \mathcal{D}_{\delta_{k + 1}}(\boldsymbol{\theta}_{j})$. The claim is that $\mathcal{P}_{\theta_{i}},\mathcal{P}_{\theta_{j}}$ is $\theta_{i}$-disconnected. 

Consider first the case where $j \leq k$. Let $\bar{\theta}_{i} \in \mathcal{D}_{\delta_{k}}(\boldsymbol{\theta}_{i})$ and $\bar{\theta}_{j} \in \mathcal{D}_{\delta_{k}}(\boldsymbol{\theta}_{j})$ be the dyadic parents of $\theta_{i},\theta_{j}$. Since $j \in \{1,\ldots,k\} \, \setminus \, \{i\}$, our inductive hypothesis \nref{I3} guarantees that the pair $\mathcal{P}_{\bar{\theta}_{i}},\mathcal{P}_{\bar{\theta}_{j}}$ is $\bar{\theta}_{i}$-disconnected. In particular, if $\ell \subset \R^{2}$ is a line with $\sigma(\ell) \in \theta_{i} \subset \bar{\theta}_{i}$, then $[\ell]_{\delta_{k}/2}$ does not intersect both $\cup \mathcal{P}_{\bar{\theta}_{i}},\cup \mathcal{P}_{\bar{\theta}_{j}}$. Since $\cup \mathcal{P}_{\theta_{i}} \subset \cup \mathcal{P}_{\bar{\theta}_{i}}$ and $\cup \mathcal{P}_{\theta_{j}} \subset \cup \mathcal{P}_{\bar{\theta}_{j}}$, we infer that $[\ell]_{\delta_{k}/2}$ also does not intersect both $\cup \mathcal{P}_{\theta_{i}},\cup \mathcal{P}_{\theta_{j}}$. Thus $\mathcal{P}_{\theta_{i}},\mathcal{P}_{\theta_{j}}$ is $\theta_{i}$-disconnected.

Consider finally the case $j = k + 1$, thus $i \in \{1,\ldots,m\} \, \setminus \, \{k + 1\}$. By the definition \eqref{form17} of the sets $\mathcal{P}_{\theta}$,
\begin{displaymath} \mathcal{P}_{\theta_{i}} = \overline{\mathcal{P}}_{\theta_{i}} \quad \text{and} \quad \mathcal{P}_{\theta_{j}} = \mathcal{P}_{\theta_{k}}' \subset \mathcal{D}_{\delta_{k + 1}}(\mathcal{P}_{\theta_{k}}), \end{displaymath} 
where $\theta_{k} \in \mathcal{D}_{\delta_{k}}(\boldsymbol{\theta}_{k + 1})$ is the dyadic parent of $\theta_{j}$. But now the pair $\overline{\mathcal{P}}_{\theta_{i}},\mathcal{P}_{\theta_{k}}'$ if $\theta_{i}$-disconnected by construction, see \eqref{form18}. Thus $\mathcal{P}_{\theta_{i}},\mathcal{P}_{\theta_{j}}$ is $\theta_{i}$-disconnected.

Finally, we verify property \nref{I4}. Fix $j \in \{1,\ldots,m\} \, \setminus \, \{k + 1\}$ and $\theta \in \mathcal{D}_{\delta_{k + 1}}(\boldsymbol{\theta}_{j})$. In this case $\mathcal{P}_{\theta}$ coincides with $\overline{\mathcal{P}}_{\theta}$ by \eqref{form17}, and property (b) yields $\diam_{\theta}(\mathcal{P}_{\theta}) \leq \Delta$.

We have now verified all the properties \nref{I1}-\nref{I4} for the scale $\delta_{k + 1}$ and the families $\mathcal{P}_{\theta} \subset \mathcal{D}_{\delta_{k + 1}}$ with $\theta \in \mathcal{D}_{\delta_{k + 1}}([0,1))$. We complete the proof of Theorem \ref{thm1} as follows. Let $\delta := \delta_{m} \in 2^{-\N} \cap (0,\Delta]$, and let $\mathcal{P}_{\theta} \subset \mathcal{D}_{\delta}$, $\theta \in \mathcal{D}_{\delta}([0,1))$, be the scale, and the families, obtained at the (final) step $m$ of the construction. We claim that these objects satisfy the claims \nref{A}-\nref{D} of Theorem \ref{thm1}. For \nref{A}, if $\boldsymbol{\theta} \in \mathcal{D}_{\Delta}([0,1))$ and $\theta \in \mathcal{D}_{\delta}(\boldsymbol{\theta})$, the inclusion $\mathcal{P}_{\theta} \subset \mathcal{D}_{\delta}(\mathcal{P}_{\boldsymbol{\theta}})$ follows from the "nestedness" property \nref{I1} applied with $i = 0$ and $j = m$.

For \nref{B}, if $\theta \in \mathcal{D}_{\delta}([0,1)) \, \setminus \, \mathcal{D}_{\delta}(\boldsymbol{\theta}_{m})$, the diameter bound $\diam_{\theta}(\mathcal{P}_{\theta}) \leq \Delta$ follows directly from property \nref{I4}. Consider then the case $\theta \in \mathcal{D}_{\delta}(\boldsymbol{\theta}_{m})$. Let $\bar{\theta} \in \mathcal{D}_{\delta_{m - 1}}(\boldsymbol{\theta}_{m})$ be the dyadic parent of $\theta$. Then $m \in \{1,\ldots,m\} \, \setminus \, \{m - 1\}$, so \nref{I4} applied with $(j,k) = (m,m - 1)$ shows that $\diam_{\bar{\theta}}(\mathcal{P}_{\bar{\theta}}) \leq \Delta$. Now $\diam_{\theta}(\mathcal{P}_{\theta}) \leq \diam_{\theta}(\mathcal{P}_{\bar{\theta}}) \leq \Delta$ by $\theta \subset \bar{\theta}$ and $\mathcal{P}_{\theta} \subset \mathcal{P}_{\bar{\theta}}$.

For \nref{C}, the lower bound $\mathcal{H}^{t - \epsilon}_{\infty}(\mathcal{P}_{\theta}) \geq \mathcal{H}^{t}_{\infty}(\mathcal{P}_{\boldsymbol{\theta}})$ for $\boldsymbol{\theta} \in \mathcal{D}_{\Delta}([0,1))$, $\theta \in \mathcal{D}_{\delta}(\boldsymbol{\theta})$, and $t \in [1 + \epsilon,\tfrac{3}{2}]$ follows from \nref{I2} applied with $k = m$. 

For \nref{D}, we need to check that if $\boldsymbol{\theta}_{i},\boldsymbol{\theta}_{j} \in \mathcal{D}_{\Delta}([0,1))$ with $i,j \in \{1,\ldots,m\}$, $i \neq j$, and $\theta_{i} \in \mathcal{D}_{\delta}(\boldsymbol{\theta}_{i})$ and $\theta_{j} \in \mathcal{D}_{\delta}(\boldsymbol{\theta}_{j})$, then $\mathcal{P}_{\theta_{i}},\mathcal{P}_{\theta_{j}}$ is $\theta_{i}$-disconnected and $\theta_{j}$-disconnected. This follows directly from \nref{I3} applied with $k = m$, and to both pairs $(i,j)$ and $(j,i)$. \end{proof}

\section{Proof of Theorem \ref{thm1}}\label{s3}

We now construct the compact set $K \subset \R^{2}$ announced in Theorem \ref{thm1}. Fix $\epsilon > 0$. We will first construct a compact set $K \subset \R^{2}$ such that $\vis^{2}_{\sigma}(K) \geq t - \epsilon$ for all $\sigma \in [0,1)$. At the end of the argument, we will explain the (standard but notationally cumbersome) edits required to get rid of the "$\epsilon$".

Write $\epsilon_{0} := 0$, and set
\begin{displaymath} \epsilon_{j} := \epsilon \cdot (\tfrac{1}{2} + \ldots + 2^{-j}) \in (0,\epsilon), \qquad j \geq 1. \end{displaymath}
We will recursively construct a sequence of dyadic scales $\delta_{0} > \delta_{1} > \delta_{2} > \ldots > 0$ and associated families $\{\mathcal{P}_{\theta} \subset \mathcal{D}_{\delta_{k}} : \theta \in \mathcal{D}_{\delta_{k}}([0,1))\}$. Set $\delta_{0} := 1$ and write $\mathcal{P}_{[0,1)} := \{[0,1)^{2}\}$. Note that 
\begin{displaymath} \mathcal{H}^{3/2}_{\infty}(\mathcal{P}_{\theta}) = \mathcal{H}^{3/2}_{\infty}([0,1)^{2}) = 1, \qquad \theta \in \mathcal{D}_{\delta_{0}}([0,1)). \end{displaymath}

Assume next that the scales $\delta_{0} > \ldots > \delta_{k}$ have already been constructed for some $k \geq 0$. Assume also that to every $j \in \{0,\ldots,k\}$ we have associated a family of the form $\{\mathcal{P}_{\theta} \subset \mathcal{D}_{\delta_{j}} : \theta \in \mathcal{D}_{\delta_{j}}([0,1))\}$, such that the following properties are valid:
\begin{itemize}
\item[(J1) \phantomsection \label{J1}] Let $k \geq 1$. If $\boldsymbol{\theta} \in \mathcal{D}_{\delta_{k - 1}}([0,1))$ and $\theta \in \mathcal{D}_{\delta_{k}}(\boldsymbol{\theta})$, then $\mathcal{P}_{\theta} \subset \mathcal{D}_{\delta_{k}}(\mathcal{P}_{\boldsymbol{\theta}})$.
\item[(J2) \phantomsection \label{J2}] Let $k \geq 1$. Then $\diam_{\theta}(\mathcal{P}_{\theta}) \leq \delta_{k - 1}$ for $\theta \in \mathcal{D}_{\delta_{k}}([0,1))$.
\item[(J3) \phantomsection \label{J3}] Let $k \geq 0$. Then $\mathcal{H}^{3/2 - \epsilon_{k}}_{\infty}(\mathcal{P}_{\theta}) \geq 1$ for all $\theta \in \mathcal{D}_{\delta_{k}}([0,1))$.
\item[(J4) \phantomsection \label{J4}] Let $k \geq 1$, $\boldsymbol{\theta},\boldsymbol{\theta}' \in \mathcal{D}_{\delta_{k - 1}}([0,1))$ with $\boldsymbol{\theta} \neq \boldsymbol{\theta}'$, and $\theta \in \mathcal{D}_{\delta_{k}}(\boldsymbol{\theta})$ and $\theta' \in \mathcal{D}_{\delta_{k}}(\boldsymbol{\theta}')$. Then the pair $\mathcal{P}_{\theta},\mathcal{P}_{\theta'}$ is $\theta$-disconnected and $\theta'$-disconnected.
\end{itemize}
Only \nref{J3} says something non-vacuous for $k = 0$, and that condition with $\epsilon_{0} = 0$ is satisfied by the initial families $\mathcal{P}_{\theta} = \{p_{\theta}\}$, $\theta \in \mathcal{D}_{\delta_{0}}([0,1))$.

We next construct the scale $\delta_{k + 1} \in 2^{-\N} \cap (0,\tfrac{1}{2}\delta_{k}]$ and the family $\{\mathcal{P}_{\theta} \subset \mathcal{D}_{\delta_{k + 1}} : \theta \in \mathcal{D}_{\delta_{k + 1}}([0,1))\}$ by applying Theorem \ref{thm1} with parameters 
\begin{displaymath} \Delta := \delta_{k} \quad \text{and} \quad \epsilon \cdot 2^{-k - 1}. \end{displaymath}
The properties \nref{J1}-\nref{J4} are nothing but restatements of Theorem \ref{thm1}\nref{A}-\nref{D}, but let us say a few words as a reminder. 

The nestedness property \nref{J1} is immediate from Theorem \ref{thm1}\nref{A}. Regarding \nref{J2}, Theorem \ref{thm1}\nref{B} tells us that $\diam_{\theta}(\mathcal{P}_{\theta}) \leq \Delta = \delta_{k}$ for $\theta \in \mathcal{D}_{\delta_{k + 1}}([0,1))$. Regarding \nref{J3}, noting that $\epsilon_{k + 1} = \epsilon_{k} + \epsilon \cdot 2^{-k - 1}$, Theorem \ref{thm1}\nref{C} tells us that
\begin{displaymath} \mathcal{H}_{\infty}^{3/2 - \epsilon_{k + 1}}(\mathcal{P}_{\theta}) = \mathcal{H}_{\infty}^{3/2 - \epsilon_{k} - \epsilon \cdot 2^{-k - 1}}(\mathcal{P}_{\theta}) \stackrel{\textup{\nref{C}}}{\geq} \mathcal{H}^{3/2 - \epsilon_{k}}(\mathcal{P}_{k}) \geq 1, \qquad \theta \in \mathcal{D}_{\delta_{k + 1}}([0,1)). \end{displaymath}
Finally, to verify \nref{J4}, fix distinct $\boldsymbol{\theta},\boldsymbol{\theta}' \in \mathcal{D}_{\delta_{k}}([0,1))$, and $\theta \in \mathcal{D}_{\delta_{k + 1}}(\boldsymbol{\theta})$ and $\theta' \in \mathcal{D}_{\delta_{k + 1}}(\boldsymbol{\theta}')$. Since we applied Theorem \ref{thm1} with $\Delta = \delta_{k}$, Theorem \ref{thm1}\nref{D} now directly tells us that $\mathcal{P}_{\theta},\mathcal{P}_{\theta'}$ is $\theta$-disconnected and $\theta'$-disconnected.

Now that the objects $\{\delta_{k}\}_{k \in \N}$ and the associated families $\{\mathcal{P}_{\theta} \subset \mathcal{D}_{\delta_{k}} : \theta \in \mathcal{D}_{\delta_{k}}([0,1))\}_{k \in \N}$ have been constructed, we use them to build a compact set $K \subset [0,1]^{2}$ with large bi-visible parts. For $k \in \N$, write
\begin{displaymath} K_{k} := \bigcup_{\theta \in \mathcal{D}_{\delta_{k}}([0,1))} \overline{\cup \mathcal{P}_{\theta}} \subset [0,1]^{2}. \end{displaymath}
Evidently each $K_{k}$ is compact. We claim that $K_{k + 1} \subset K_{k}$ for all $k \in \N$. This follows from \nref{J1}: if $\theta \in \mathcal{D}_{\delta_{k + 1}}([0,1))$, then $\mathcal{P}_{\theta} \subset \mathcal{D}_{\delta_{k + 1}}(\mathcal{P}_{\boldsymbol{\theta}})$, where $\boldsymbol{\theta} \in \mathcal{D}_{\delta_{k}}([0,1))$ is the dyadic parent of $\theta$. Consequently 
\begin{displaymath} \overline{\cup \mathcal{P}_{\theta}} \subset \overline{\cup \mathcal{P}_{\boldsymbol{\theta}}} \subset K_{k}, \qquad \theta \in \mathcal{D}_{\delta_{k + 1}}([0,1)), \end{displaymath}
which gives $K_{k + 1} \subset K_{k}$. We now define the compact non-empty set 
\begin{displaymath} K := \bigcap_{k \geq 0} K_{k} \subset [0,1]^{2}. \end{displaymath}
Next, we define a compact subset $K(\sigma) \subset K$ for each $\sigma \in [0,1)$. Fix $\sigma \in [0,1)$, and let $\{\theta_{k}(\sigma)\}_{k \in \N}$ be the sequence of dyadic arcs satisfying $\sigma \in \theta_{k}(\sigma) \in \mathcal{D}_{\delta_{k}}([0,1))$, $k \in \N$. Define
\begin{equation}\label{form20} K_{k}(\sigma) := \overline{\cup \mathcal{P}_{\theta_{k}(\sigma)}}, \qquad \sigma \in [0,1), \, k \in \N. \end{equation}
Then $K_{k + 1}(\sigma) \subset K_{k}(\sigma)$ by \nref{J1} (same argument that proved $K_{k + 1} \subset K_{k}$). Therefore
\begin{displaymath} K(\sigma) := \bigcap_{k \geq 0} K_{k}(\sigma), \qquad \sigma \in [0,1), \end{displaymath}
is a non-empty compact set. Furthermore $K(\sigma) \subset K$, since $K_{k}(\sigma) \subset K_{k}$ for all $k \geq 0$.

Next, we claim that $K(\sigma) \subset \vis^{2}_{\sigma}(K)$ for all $\sigma \in [0,1)$. Fix $\sigma \in [0,1)$ and $z \in K(\sigma)$. Let $\ell := \ell_{z,\sigma}$ (the line with slope $\sigma$ through $z$). We claim that 
\begin{equation}\label{form19} \diam(\ell \cap K) \leq \diam(\ell \cap K_{k}) \leq \delta_{k - 2}, \qquad k \geq 2. \end{equation}
(The first inequality is trivial.) The second inequality will show, letting $k \to \infty$, that $\diam(\ell \cap K) = 0$, verifying that $z \in \vis_{\sigma}^{2}(K)$.

To prove \eqref{form19}, fix $k \geq 2$, and let $\boldsymbol{\theta} \in \mathcal{D}_{\delta_{k - 1}}([0,1))$ and $\theta \in \mathcal{D}_{\delta_{k}}([0,1))$ be the $\delta_{k - 1}$-arc and $\delta_{k}$-arc containing $\sigma$, respectively. Then
\begin{equation}\label{form21} z \in K(\sigma) \subset K_{k}(\sigma) \stackrel{\eqref{form20}}{=} \overline{\cup \mathcal{P}_{\theta}}. \end{equation}
Now comes a central observation. Let $\theta' \in \mathcal{D}_{\delta_{k}}([0,1)) \, \setminus \, \mathcal{D}_{\delta_{k}}(\boldsymbol{\theta})$. Then $\mathcal{P}_{\theta},\mathcal{P}_{\theta'}$ is $\theta$-disconnected by \nref{J4}. Since $\ell$ is a line with slope in $\theta$, and $[\ell]_{\delta_{k}/2}$ intersects $\cup \mathcal{P}_{\theta}$ by \eqref{form21}, we see that $[\ell]_{\delta_{k}/2} \cap (\cup \mathcal{P}_{\theta'}) = \emptyset$. In particular
\begin{displaymath} \ell \cap \left(\overline{\cup \mathcal{P}_{\theta'}}\right) = \emptyset, \qquad \theta' \in \mathcal{D}_{\delta_{k}}([0,1)) \, \setminus \, \mathcal{D}_{\delta}(\boldsymbol{\theta}). \end{displaymath}
Therefore
\begin{displaymath} K_{k} \cap \ell \subset \bigcup_{\theta' \in \mathcal{D}_{\delta_{k}}(\boldsymbol{\theta})} \overline{\cup \mathcal{P}_{\theta'}} \, \stackrel{\textup{\nref{J1}}}{\subset} \, \overline{\cup \mathcal{P}_{\boldsymbol{\theta}}}. \end{displaymath}
Recall that $\diam_{\boldsymbol{\theta}}(\mathcal{P}_{\boldsymbol{\theta}}) \leq \delta_{k - 2}$ by \nref{J2}. Since $\ell$ is a line with slope $\sigma(\ell) \in \boldsymbol{\theta}$, 
\begin{displaymath} \diam(\overline{\cup \mathcal{P}_{\boldsymbol{\theta}}} \cap \ell) \leq \diam(\cup \mathcal{P}_{\boldsymbol{\theta}} \cap [\ell]_{\delta_{k - 1}/2}) \leq \delta_{k - 2}. \end{displaymath}
Combining the previous two displayed equations proves \eqref{form19}.

We have now shown that $K(\sigma) \subset \vis^{2}_{\sigma}(K)$ for all $\sigma \in [0,1)$. It remains to prove that $\Hd K(\sigma) \geq \tfrac{3}{2} - \epsilon$ for all $\sigma \in [0,1)$. Let $\{U_{i}\}_{i \in \N}$ be a cover of $K(\sigma)$ by bounded open sets. Since $K(\sigma)$ is compact,
\begin{displaymath} \dist\Big(K(\sigma),\R^{2} \, \setminus \, \bigcup_{i \in \N} U_{i} \Big) =: \eta > 0. \end{displaymath}
 Therefore $K_{k}(\sigma) \subset \cup \{U_{i}\}_{i \in \N}$ for all $k \in \N$ so large that $\delta_{k} \leq \eta/2$. For any such $k \in \N$,
\begin{equation}\label{form22} \sum_{i \in \N} \diam(U_{i})^{3/2 - \epsilon} \gtrsim \mathcal{H}^{3/2 - \epsilon}_{\infty}(K_{k}(\sigma)) \geq \mathcal{H}_{\infty}^{3/2 - \epsilon}(\mathcal{P}_{\theta_{k}(\sigma)}) \stackrel{\textup{\nref{J3}}}{\geq} 1.\end{equation} 
Therefore $\mathcal{H}^{3/2 - \epsilon}(K(\sigma)) > 0$, and $\Hd K(\sigma) \geq 3/2 - \epsilon$ for all $\sigma \in [0,1)$.

Let us finally explain how the construction is modified to get the stronger conclusion $\Hd \vis^{2}_{\sigma}(K) \geq 3/2$ for all $\sigma \in [0,1)$. Recall the construction of the sequences $\{\delta_{k}\}_{k \in \N}$ and the families $\{\mathcal{P}_{\theta} : \theta \in \mathcal{D}_{\delta_{k}}([0,1))\}_{k \in \N}$. At an arbitrary stage $k(1) \in \N$, it holds $\mathcal{H}^{3/2}_{\infty}(\mathcal{P}_{\theta}) \geq c_{1} > 0$ for all $\theta \in \mathcal{D}_{\delta_{k(1)}}([0,1))$, where $c_{1}$ is the $(3/2)$-dimensional (dyadic) Hausdorff content of a single $\delta_{k(1)}$-square. From this point on, we may continue the construction in such a way that simultaneously 
\begin{displaymath} \mathcal{H}^{3/2 - \epsilon_{k}}_{\infty}(\mathcal{P}_{\theta'}) \geq 1 \quad \text{and} \quad \mathcal{H}^{3/2 - \epsilon_{k}/2}_{\infty}(\mathcal{P}_{\theta'}) \geq c_{1}, \qquad \theta' \in \mathcal{D}_{\delta_{k}}([0,1)), \, k \geq k(1). \end{displaymath}
(This simultaneous control is possible because Theorem \ref{thm1}\nref{C} holds for all $t \in [1 + \epsilon,\tfrac{3}{2}]$.) We repeat this trick countably many times during the construction. For every $n \in \N$, we designate an index $k(n) \in \N$, and a (decreasing) sequence of lower bounds $c_{1}\ldots,c_{n}> 0$ such that 
\begin{displaymath} \mathcal{H}_{\infty}^{3/2 - \epsilon_{k}}(\mathcal{P}_{\theta}) \geq 1, \, \mathcal{H}_{\infty}^{3/2 - \epsilon_{k}/2}(\mathcal{P}_{\theta}) \geq c_{1}, \quad \ldots \quad \mathcal{H}_{\infty}^{3/2 - \epsilon_{k}/2^{n}}(\mathcal{P}_{\theta}) \geq c_{n}, \end{displaymath}
for all $\theta \in \mathcal{D}_{\delta_{k}}([0,1))$ with $k \geq k(n)$. Finally the estimate \eqref{form22} will show that $\mathcal{H}^{t - \epsilon/2^{n}}(K(\sigma)) \gtrsim c_{n}$ for every $n \in \N$, and $\sigma \in [0,1)$. Therefore $\Hd K(\sigma) \geq 3/2$ for all $\sigma \in [0,1)$. 

\appendix

\section{Basic building block}\label{appA}

The set $\mathfrak{P}$ and the tubes $\mathcal{T}_{\sigma}$ claimed in Proposition \ref{prop1} can be obtained by "dualising" a suitable grid construction (Figure \ref{fig1}) which was used to construct the main examples in \cite{MR4745881}. We start by recalling the details of that construction from \cite[Appendix A.1]{MR4745881}, and then we will show to transform this object into the one claimed in Proposition \ref{prop1}.

We define a piece of notation. Recall that $\pi_{\zeta}(x,y) := \zeta x + y$ for $\zeta \in [0,1)$ and $(x,y) \in \R^{2}$. If $\mathcal{P} \subset \mathcal{D}_{\delta}$ is a family of dyadic $\delta$-squares, $\pi_{0}(\cup \mathcal{P})$ is a family of dyadic $\delta$-intervals. We write $Y_{\mathcal{P}} \subset \delta \Z$ for the left-endpoints of these intervals. 

\begin{lemma}\label{lemma1} Let $\tau \in (1,2]$, $s \in (0,2 - \tau]$, and $\Delta \in 2^{-\N}$. For all $\delta \in 2^{-\N} \cap (0,\Delta]$ sufficiently small in terms of $\Delta,s,\tau$, there exists a set $\mathcal{P} \subset \mathcal{D}_{\delta}([-3,3]^{2})$ with the following properties.
\begin{itemize}
\item[\textup{(P1)} \phantomsection \label{P1}] If $\sigma \in \delta \Z \cap [0,1)$, then the "vertical slice" $\{p \in \mathcal{P} : p \cap (\{\sigma\} \times \R) \neq \emptyset\}$ contains a $\sim \delta^{\tau - 1}$-separated subset $\mathcal{P}_{\sigma}$ with the following property: if $I \subset [-2,2]$ is an interval of length $\ell(I) \geq \Delta$, then 
\begin{equation}\label{form24} |Y_{\mathcal{P}_{\sigma}} \cap I| \sim \ell(I)\delta^{1 -\tau}. \end{equation}

\item[\textup{(P2)} \phantomsection \label{P2}] There exists a $\delta^{s}$-separated subset $\Sigma \subset \delta \Z \cap [0,1]$ such that $|\pi_{\zeta}(\mathcal{P})|_{\delta} \lesssim \delta^{-(s + \tau)/2}$ for all $\zeta \in \Sigma$, and such that $\Sigma$ has the following property. If $I \subset [0,1]$ is an interval of length $\ell(I) \geq \Delta$, then $|\Sigma \cap I| \sim \ell(I)\delta^{-s}$.
\end{itemize}
\end{lemma}

The family $\mathcal{P} \subset \mathcal{D}_{\delta}([-3,3]^{2})$ in Lemma \ref{lemma1} is illustrated in Figure \ref{fig1}, reproduced from \cite[Appendix A.1]{MR4745881}. The family $\mathcal{P}$ is obtained by slightly rotating a grid of $\delta$-squares of cardinality $\delta^{-\tau}$ in such a way that every vertical "slice" contains the expected number $\sim \delta^{1 - \tau}$ of $\delta$-squares.

\begin{figure}[h!]
\begin{center}
\begin{overpic}[scale = 0.6]{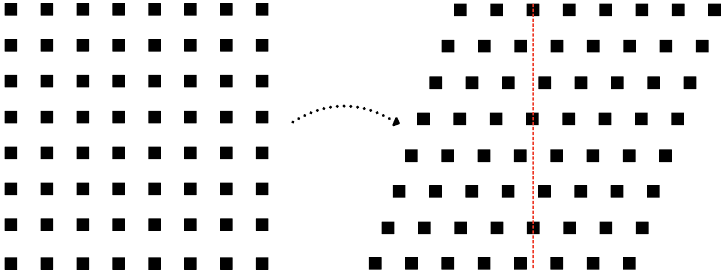}
\put(41,25){\tiny{rotation}}
\end{overpic}
\caption{The set in Lemma \ref{lemma1}.}\label{fig1}
\end{center}
\end{figure}

Slightly coarser versions of \nref{P1}-\nref{P2} are stated in \cite[(P1)-(P2)]{MR4745881}; the level precision in \nref{P1}-\nref{P2} was not needed in \cite{MR4745881}. However, the proofs in \cite[Appendix A.1]{MR4745881} give exactly what we need in \nref{P1}-\nref{P2}. We explain this in the following two remarks.

\begin{remark} Regarding \nref{P1}, it seems likely that one can simply take 
\begin{displaymath} \mathcal{P}_{\sigma} := \{p \in \mathcal{P} : p \cap (\{\sigma\} \times \R) \neq \emptyset\}, \end{displaymath}
but the existence of the claimed subset is directly guaranteed by \cite[Claim A.2]{MR4745881}. The proof of \cite[Claim A.2]{MR4745881} shows that $\{p \in \mathcal{P} : p \cap (\{\sigma\} \times \R) \neq \emptyset\}$ contains an arithmetic progression with gap $\delta^{\tau - 1}$ and length $\sim \delta^{1 - \tau}$. Now $\mathcal{P}_{\sigma}$ can be defined to be this progression, and the "density property" \eqref{form24} holds for $\delta > 0$ so small that $2\delta^{\tau - 1} \leq \Delta$.  \end{remark} 

\begin{remark} The set $\Sigma$ in \nref{P2} is defined to be a maximal $\delta^{s}$-separated subset of the family "$\mathcal{E}$" appearing in \cite[(A.3)]{MR4745881}. Now \cite[Claim A.5]{MR4745881} states that $|\Sigma \cap I| \sim \ell(I)\delta^{-s}$ for all intervals $I \subset [0,1]$ of length $\ell(I) \geq C\delta^{s/2}\log(1/\delta)$, where $C \geq 1$ is absolute. In particular, \nref{P2} holds provided that $\delta > 0$ is so small that $C\delta^{s/2}\log(1/\delta) \leq \Delta$. The upper bound $|\pi_{\zeta}(\mathcal{P})|_{\delta} \lesssim \delta^{-(s + \tau)/2}$, for $\zeta \in \Sigma$, is established in \cite[Claim A.4]{MR4745881}. \end{remark}

We are then ready to prove Proposition \ref{prop1}, which we restate here:
\begin{proposition}\label{prop1App} For every $\epsilon,\Delta \in 2^{-\N}$, there exists $\delta_{0} = \delta_{0}(\epsilon,\Delta) \in (0,\tfrac{1}{2}\Delta]$ such that the following holds for all $\delta \in 2^{-\N} \cap (0,\delta_{0}]$. There exists a family $\mathfrak{P} \subset \mathcal{D}_{\delta}$ with the following property. Assume that $t \in [1 + \epsilon,\tfrac{3}{2}]$, and $\mathcal{Q} \subset \mathcal{D}_{\Delta}$. Then 
\begin{equation}\label{form3app} \mathcal{H}_{\infty}^{t - \epsilon}(\mathcal{Q} \cap \mathfrak{P}) \geq \mathcal{H}^{t}_{\infty}(\mathcal{Q}). \end{equation}
Moreover, for each $\sigma \in \delta \Z \cap [0,1)$ there exists a family $\mathcal{R}_{\sigma}$ of $(\delta \times \Delta)$-rectangles whose longer sides have slope $\sigma$, with the following properties:
\begin{itemize}
\item[\textup{(1)}] Let $\mathcal{T}_{\sigma}$ be the $\delta$-tubes spanned by the rectangles $\mathcal{R}_{\sigma}$. Every tube in $\mathcal{T}_{\sigma}$ contains exactly one element of $\mathcal{R}_{\sigma}$, and the tubes in $\mathcal{T}_{\sigma}$ are $10\delta$-separated.
\item[\textup{(2)}] $\cup \mathfrak{P} \cap (\cup \{10T : T \in \mathcal{T}_{\sigma}\}) = \emptyset$.
\item[\textup{(3)}] If $t \in [1 + \epsilon,\tfrac{3}{2}]$, and $\mathcal{Q} \subset \mathcal{D}_{\Delta}$  then 
\begin{displaymath} \mathcal{H}^{t - \epsilon}_{\infty}(\mathcal{Q} \cap \mathcal{D}_{\delta}(\cup \mathcal{R}_{\sigma})) \geq \mathcal{H}^{t}_{\infty}(\mathcal{Q}). \end{displaymath}
\end{itemize}
\end{proposition}

\begin{proof} Fix $\epsilon,\Delta \in 2^{-\N}$ as in the statement (we may assume that $\epsilon \in (0,1)$). Fix also $\underline{\Delta} \in (0,\tfrac{1}{100}\Delta]$ so small that
\begin{equation}\label{form9} \underline{\Delta}^{-\epsilon/2} \geq \max\{C,\Delta^{-1}\}, \end{equation} 
where $C > 0$ is an absolute constant to be determined later. Then, we apply Lemma \ref{lemma1} with this $\underline{\Delta}$,
\begin{displaymath} \tau := \tfrac{3}{2} \quad \text{and} \quad s := (1 - \epsilon)/2 \in (0,2 - \tau], \end{displaymath}
and $\delta \in (0,\underline{\Delta}]$ so small that the conditions of Lemma \ref{lemma1} are satisfied, and additionally
\begin{equation}\label{form7} C\delta^{\epsilon/2} \leq \underline{\Delta}^{4}/2. \end{equation}
(The largest "$\delta$" so that the previous conditions are met is now defined to be the scale threshold "$\delta_{0}$" in Proposition \ref{prop1App}.) Let $\mathcal{P} \subset \mathcal{D}_{\delta}([-3,3]^{2})$ be the resulting set.

\begin{figure}[h!]
\begin{center}
\begin{overpic}[scale = 0.4]{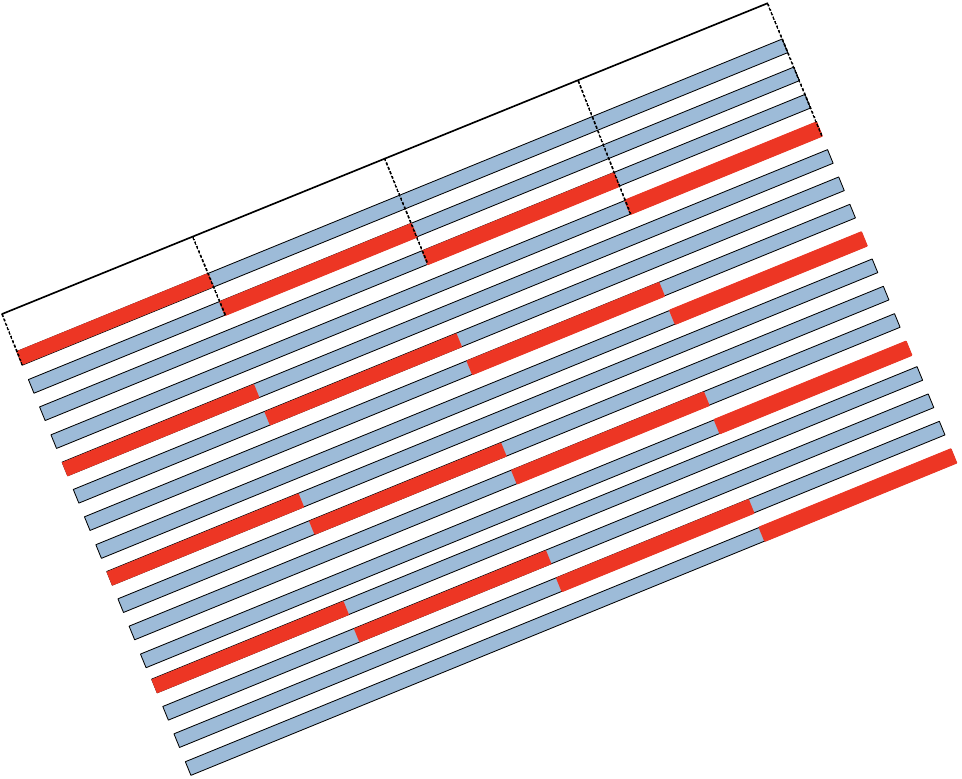}
\put(8,55){$I_{1}$}
\put(25,62.5){$I_{2}$}
\put(45,70){$I_{3}$}
\put(66,79){$I_{4}$}
\put(-2,42){\tiny{$b_{1}$}}
\put(-1,39){\tiny{$b_{2}$}}
\put(0,36){\tiny{$b_{3}$}}
\put(1,33){\tiny{$b_{4}$}}
\put(2,30){\tiny{$b_{5}$}}
\put(3,27){\tiny{$b_{6}$}}
\put(4,24){\tiny{$\ldots$}}
\put(13,-1){\tiny{$b_{16}$}}

\end{overpic}
\caption{The choice of the rectangles $R_{T}$, $T \in \mathcal{T}_{\sigma}$, shown in red. For artistic reasons the picture shows the rectangles $T_{\sigma,b_{j}} \cap \bar{\pi}_{\sigma}^{-1}(I_{i})$; the real ones defined in \eqref{def:rectangles} have twice the length.}\label{fig3}
\end{center}
\end{figure}
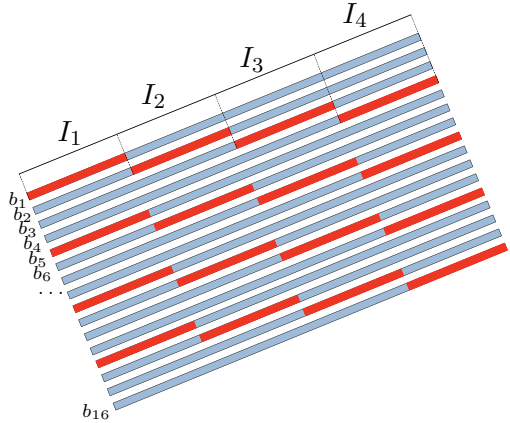

We now define the families of rectangles and tubes appearing in the statement. Fix $\sigma \in \delta \Z \cap [0,1)$, and recall the (subset of the) "vertical slice" $\mathcal{P}_{\sigma}$ from Lemma \ref{lemma1}\nref{P1}. We define the $\delta$-tube family $\mathcal{T}_{\sigma}$ as the "dual" of $\mathcal{P}_{\sigma}$, under standard point-line duality. Abbreviate $Y_{\sigma} := Y_{\mathcal{P}_{\sigma}}$ (recall the notation from Lemma \ref{lemma1}), and define the family of (parallel slope-$\sigma$) lines
\begin{displaymath} \mathcal{L}_{\sigma} := \{\ell_{\sigma,b} : b \in Y_{\sigma}\}, \end{displaymath}
where $\ell_{\sigma,b} = \{(x,y) \in \R^{2} : y = \sigma x + b\}$. Finally, define the $\delta$-tubes 
\begin{displaymath} \mathcal{T}_{\sigma} := \{T_{\sigma,b} : b \in Y_{\sigma}\} :=  \{[\ell_{\sigma,b}]_{\delta/2} : b \in Y_{\sigma}\}. \end{displaymath}
Note that the lines in $\mathcal{L}_{\sigma}$ have roughly the same separation as the squares in $\mathcal{P}_{\sigma}$, which is $\sim \delta^{\tau - 1} = \delta^{1/2}$ according to Lemma \ref{lemma1}\nref{P1}. In particular, the tubes in $\mathcal{T}_{\sigma}$ are $10\delta$-separated (assuming $\delta > 0$ small enough). This proves Proposition \ref{prop1App}\nref{1}.

For each $T \in \mathcal{T}_{\sigma}$, we plan to pick a single $(\delta \times \Delta)$-rectangle $R_{T} \subset T$. Once this has been accomplished, we will define the family of $(\delta \times \Delta)$-rectangles appearing in the statement of Proposition \ref{prop1App} as
\begin{displaymath} \mathcal{R}_{\sigma} := \{R_{T} : T \in \mathcal{T}_{\sigma}\}. \end{displaymath}
Likely one (working) way of choosing the rectangles $R_{T}$ would be a uniformly random selection among all those $(\delta \times \Delta)$-subrectangles of $T$ which hit $[0,1]^{2}$ (or an arbitrary rectangle if $T \cap [0,1]^{2} = \emptyset$). For technical reasons, a deterministic choice is slightly preferable, where the rectangles are chosen as "well-spread" as possible. The idea can be deciphered from Figure \ref{fig3}, but also need an analytic expression to work with.

Let $\bar{\pi}_{\sigma}$ be the orthogonal projection to the subspace parallel to lines in $\mathcal{L}_{\sigma}$ (or the tubes $\mathcal{T}_{\sigma}$). Enumerate $\mathcal{D}_{\Delta/2}([-2,2]) =: \{I_{1},\dots,I_{m}\}$, where $m \sim \Delta^{-1}$. Enumerate also $Y_{\sigma} := \{b_{1},\ldots,b_{n}\}$ in increasing order, where $n = |Y_{\sigma}| \sim \delta^{-1/2}$. Note that $n \gg m$. Now, for each $b_{j} \in Y_{\sigma}$, define the $(\delta \times \Delta)$-rectangle
\begin{equation}\label{def:rectangles} R_{\sigma,b_{j}} := T_{\sigma,b_{j}} \cap \bar{\pi}_{\sigma}^{-1}(2I_{i}) \subset T_{\sigma,b_{j}}, \qquad j = i (\mathrm{mod\,} m). \end{equation}

Now the tubes and rectangles in the statement of Proposition \ref{prop1App} have been defined. Next, we define the family $\mathfrak{P} \subset \mathcal{D}_{\delta}$, and verify the (remaining) properties \eqref{form3app} and (2)-(3). Let
\begin{displaymath} \mathcal{T} := \bigcup_{\sigma \in \delta \Z \cap [0,1)} \mathcal{T}_{\sigma}. \end{displaymath}
Thus, $\mathcal{T}$ is roughly the "tube dual" of of $\mathcal{P}$, or more precisely the dual of the subset obtained as the union of the sets $\mathcal{P}_{\sigma}$, $\sigma \in \delta \Z \cap [0,1)$. Recall from Lemma \ref{lemma1}\nref{P2} that $|\pi_{\zeta}(\mathcal{P})|_{\delta} \lesssim \delta^{-(s + \tau)/2} = \delta^{-1 + \epsilon/2}$ for all $\zeta \in \Sigma$. Write $10\mathcal{T} := \{[\ell]_{5\delta} : [\ell]_{\delta/2} \in \mathcal{T}\}$. For $\zeta \in [0,1]$, write also $(10\mathcal{T})_{\zeta}$ for the "vertical $\zeta$-slice" of the tube family $10\mathcal{T}$:
\begin{displaymath} (10\mathcal{T})_{\zeta} := \{q = [\zeta,\zeta + \delta) \times [b,b + \delta) \in \mathcal{D}_{\delta} : q \cap 10T \neq \emptyset \text{ for some } T \in \mathcal{T}\}. \end{displaymath}
A standard (easy to check) fact about point-line duality is the following: the $\pi_{\zeta}$-projection of $P \subset \R^{2}$ equals the $y$-coordinates of the vertical $\zeta$-slice of the dual line family $\{\ell_{a,b} : (a,b) \in P\}$. In our situation, this principle implies that the $y$-coordinates of the squares in $(10\mathcal{T})_{\sigma}$ are contained in the $O(\delta)$-neighbourhood of $\pi_{\zeta}(\mathcal{P})$. In particular,
\begin{equation}\label{form6} |(10\mathcal{T})_{\zeta}| \leq C|\pi_{\zeta}(\mathcal{P})|_{\delta} \leq C\delta^{-1 + \epsilon/2}, \qquad \zeta \in \Sigma, \end{equation}
where $C > 0$ is absolute. We are now prepared to define the set $\mathfrak{P}$:
\begin{equation}\label{form5} \mathfrak{P} := \bigcup_{\zeta \in \Sigma} \{q = [\zeta,\zeta + \delta) \times [b,b + \delta) \in \mathcal{D}_{\delta} : q \notin (10\mathcal{T})_{\zeta}\}. \end{equation} 
In English, we include to $\mathfrak{P}$ all the dyadic $\delta$-squares around the vertical segments $\{\zeta\} \times [0,1]$, $\zeta \in \Sigma$, which are not contained in the $10\delta$-neighbourhood of any tube in $\mathcal{T}$.

Now the set $\mathfrak{P}$ and the families $\mathcal{R}_{\sigma},\mathcal{T}_{\sigma}$ have been defined, so it remains to prove the claimed properties in Proposition \ref{prop1App}, namely \eqref{form3app} and (3). The property (2) is already clear from the construction of $\mathfrak{P}$. 

We start with \eqref{form3app}. Let $\mathcal{Q} \subset \mathcal{D}_{\Delta}$ be arbitrary, and fix $t \in [1 + \epsilon,\tfrac{3}{2}]$. The claim is that $\mathcal{H}_{\infty}^{t - \epsilon}(\mathcal{Q} \cap \mathfrak{P}) \geq \mathcal{H}^{t}_{\infty}(\mathcal{Q})$. Since $\delta \leq \underline{\Delta} \leq \Delta$, this claim is formally equivalent to the claim $\mathcal{H}^{t - \epsilon}(\mathcal{D}_{\underline{\Delta}}(\mathcal{Q}) \cap \mathfrak{P}) \geq \mathcal{H}^{t}_{\infty}(\mathcal{D}_{\underline{\Delta}}(\mathcal{Q}))$. Replacing $\mathcal{Q}$ by $\mathcal{D}_{\underline{\Delta}}(\mathcal{Q})$ without altering notation, we will from now on assume that $\mathcal{Q} \subset \mathcal{D}_{\underline{\Delta}}$.

Let $\mathcal{W} \subset \mathcal{D}$ be an arbitrary cover of $\mathcal{Q} \cap \mathfrak{P}$. We need to show that $\sum_{W \in \mathcal{W}} \ell(W)^{t - \epsilon} \geq \mathcal{H}^{t}(\mathcal{Q})$. Reducing $\mathcal{W}$ to its maximal elements, we may assume that $\mathcal{W}$ is disjoint. We may also assume that every $W \in \mathcal{W}$ intersects $Q \cap \mathfrak{P}$ for some $Q \in \mathcal{Q}$. Since $\mathfrak{P} \subset \mathcal{D}_{\delta}$, and $t < 2$, we may finally assume that $\ell(W) \geq \delta$ for all $W \in \mathcal{W}$ (as explained after \eqref{def:content}).

Let $\mathcal{Q}_{0} \subset \mathcal{Q} \subset \mathcal{D}_{\underline{\Delta}}$ be the squares in $\mathcal{Q}$ which are not strictly contained in any element of $\mathcal{W}$. We claim that
\begin{equation}\label{form4} \mathop{\sum_{W \in \mathcal{W}}}_{W \subset Q} \ell(W)^{t - \epsilon} \geq \ell(Q)^{t}, \qquad Q \in \mathcal{Q}_{0}, \, t \in [1 + \epsilon,\tfrac{3}{2}]. \end{equation} 
Let us see how this implies $\mathcal{H}^{t - \epsilon}_{\infty}(\mathcal{Q} \cap \mathfrak{P}) \geq \mathcal{H}^{t}_{\infty}(\mathcal{Q})$. Let 
\begin{displaymath} \mathcal{W}' := \mathcal{Q}_{0} \cup \mathcal{W}'', \end{displaymath}
where $\mathcal{W}''$ consists of all the squares in $\mathcal{W}$ which are not contained in any element of $\mathcal{Q}$. We claim that $\mathcal{W}'$ is a cover of $\mathcal{Q}$. Indeed, if $Q \in \mathcal{Q}_{0}$, then $Q$ is evidently covered by $\mathcal{W}'$. On the other hand, if $Q \in \mathcal{Q} \, \setminus \, \mathcal{Q}_{0}$, then by definition $Q$ is strictly contained in some $W \in \mathcal{W}$. Then $\ell(W) > \ell(Q) = \underline{\Delta}$, so $W$ cannot be contained in any element of $\mathcal{Q}$, hence $W \in \mathcal{W}''$. Thus $Q \subset W \in \mathcal{W'}$. 

Now we know that $\mathcal{W}'$ is a cover of $\mathcal{Q}$. We next claim that that every element of $\mathcal{W}$ either lies in $\mathcal{W}''$, or then is contained in some element of $\mathcal{Q}_{0}$ (these two subsets of $\mathcal{W}$ are evidently disjoint). Indeed, if $W \in \mathcal{W}$, then $W \cap \mathcal{Q} \cap \mathfrak{P} \neq \emptyset$ for some $Q \in \mathcal{Q}$. Now, if $Q \subsetneq W$, it holds $W \in \mathcal{W}''$ ($\ell(W) > \ell(Q) = \underline{\Delta}$, so $W$ cannot be contained in any element of $\mathcal{Q}$). Alternatively, $W \subset Q$. But now $Q \in \mathcal{Q}_{0}$, because $Q$ cannot be strictly contained in any element of $\mathcal{W}$: if $Q$ was strictly contained in some $W' \in \mathcal{W}$, then $W \subsetneq W'$, violating the disjointness of $\mathcal{W}$. Thus $W \subset Q \in \mathcal{Q}_{0}$.

Based on the previous claim, we can now decompose
\begin{displaymath} \sum_{W \in \mathcal{W}} \ell(W)^{t - \epsilon} = \sum_{W \in \mathcal{W}''} \ell(W)^{t - \epsilon} + \sum_{Q \in \mathcal{Q}_{0}} \mathop{\sum_{W \in \mathcal{W}}}_{W \subset Q} \ell(W)^{t - \epsilon} \stackrel{\eqref{form4}}{\geq} \sum_{W \in \mathcal{W}''} \ell(W)^{t} + \sum_{Q \in \mathcal{Q}_{0}} \ell(Q)^{t}. \end{displaymath} 
Since $\mathcal{Q}_{0} \cup \mathcal{W}'' = \mathcal{W}'$ is a cover of $\mathcal{Q}$, the right hand side is bounded from below by $\mathcal{H}^{t}_{\infty}(\mathcal{Q})$. This completes the proof of $\mathcal{H}^{t - \epsilon}_{\infty}(\mathcal{Q} \cap \mathfrak{P}) \geq \mathcal{H}^{t}_{\infty}(\mathcal{Q})$, assuming \eqref{form4}.

We next prove \eqref{form4}. Fix $Q := I \times J \in \mathcal{Q}_{0}$, where $I,J \in \mathcal{D}_{\underline{\Delta}}([0,1)$. Let us first record some properties of $Q \cap \mathfrak{P}$. By definition, recall \eqref{form5}, $\mathfrak{P}$ consists of certain squares of the form $[\zeta,\zeta + \delta) \times [b,b + \delta) \subset [0,1)^{2}$, where $\zeta \in \Sigma$ and $b \in \delta \Z$. By Lemma \ref{lemma1}\nref{P2}, the set $\Sigma$ is $\delta^{s}$-separated, and $|\Sigma \cap I| \sim \underline{\Delta} \delta^{-s}$. This implies that 
\begin{equation}\label{form8} |\mathfrak{P} \cap W| \lesssim \begin{cases} \ell(W)^{2} \delta^{-1 - s}, & \delta^{s} \leq \ell(W) \leq \underline{\Delta}, \\ \ell(W)\delta^{-1}, & \delta \leq \ell(W) \leq \delta^{s}, \end{cases} \qquad W \in \mathcal{D}. \end{equation} 
We also need a lower bound for $|\mathfrak{P} \cap Q|$. Recall from \eqref{form6} that $|(10\mathcal{T})_{\zeta}| \leq C\delta^{-1 + \epsilon/2}$ for all $\zeta \in \Sigma$. Therefore, 
\begin{align} |\mathfrak{P} \cap Q| & \geq |\Sigma \cap I| \cdot ((\underline{\Delta}/\delta) - |(10\mathcal{T})_{\zeta}|) \notag\\
&\label{form12} \gtrsim \underline{\Delta} \delta^{-s} \cdot ((\underline{\Delta}/\delta) - C\delta^{-1 + \epsilon/2}) \stackrel{\eqref{form7}}{\gtrsim} \underline{\Delta}^{2}\delta^{-1 - s}. \end{align} 
Now, to establish \eqref{form4}, note that $\mathfrak{P} \cap Q$ is covered by the squares $W \in \mathcal{W}$ with $W \subset Q$ (otherwise $Q$ would be strictly contained in some element of $\mathcal{W}$ contrary to the definition of $Q \in \mathcal{Q}_{0}$). The argument splits to two cases: (i) at least half of the squares in $\mathfrak{P} \cap Q$ are contained in squares $W \in \mathcal{W}$ with $\delta^{s} \leq \ell(W) \leq \underline{\Delta}$, or (ii) at least half of the squares in $\mathfrak{P} \cap Q$ are contained in squares $W \in \mathcal{W}$ with $\delta \leq \ell(W) \leq \delta^{s}$.

Assume that case (i) occurs. Then, using $t - \epsilon - 2 \leq 0$, we estimate
\begin{displaymath} \mathop{\sum_{W \in \mathcal{W}}}_{W \subset Q} \ell(W)^{t - \epsilon} \geq \underline{\Delta}^{t - \epsilon - 2}\delta^{1 + s} \mathop{\sum_{W \in \mathcal{W}}}_{W \subset Q, \delta^{s} \leq \ell(W) \leq \underline{\Delta}} \ell(W)^{2}\delta^{-1 - s} \stackrel{\eqref{form8}-\eqref{form12}}{\gtrsim} \underline{\Delta}^{t - \epsilon} = \underline{\Delta}^{-\epsilon}\ell(Q)^{t}. \end{displaymath}
This gives \eqref{form4}, provided that the constant "$C$" in \eqref{form9} was chosen sufficiently large.

Assume next that case (ii) occurs. This time, using $1 + \epsilon \leq t \leq \tfrac{3}{2}$, we estimate
\begin{align*} \mathop{\sum_{W \in \mathcal{W}}}_{W \subset Q} \ell(W)^{t - \epsilon} & \geq \delta^{t - \epsilon} \mathop{\sum_{W \in \mathcal{W}}}_{W \subset Q, \delta \leq \ell(W) \leq \delta^{s}} \ell(W)\delta^{-1}\\
& \stackrel{\eqref{form8}-\eqref{form12}}{\gtrsim} \delta^{t - \epsilon} \cdot \underline{\Delta}^{2}\delta^{- s - 1} = \underline{\Delta}^{2}\delta^{t - \epsilon/2 - 3/2} \geq \underline{\Delta}^{2}\delta^{-\epsilon/2}. \end{align*}
Since $\underline{\Delta}^{2}\delta^{-\epsilon/2} \geq C\underline{\Delta}^{t}$ by \eqref{form7}, this yields \eqref{form4} in case (ii). The proof of \eqref{form4} (and \eqref{form3app}) is therefore complete.

We move on to the proof of Proposition \ref{prop1App}(3). Fix $\sigma \in \delta \Z \cap [0,1)$, $\mathcal{Q} \subset \mathcal{D}_{\Delta}$ and $t \in [1 + \epsilon,\tfrac{3}{2}]$. The claim is that $\mathcal{H}^{t - \epsilon}_{\infty}(\mathcal{Q} \cap \mathcal{D}_{\delta}(\cup \mathcal{R}_{\sigma}))) \geq \mathcal{H}^{t}_{\infty}(\mathcal{Q})$. The proof greatly resembles the proof of \eqref{form3app}. Again, we may assume that $\mathcal{Q} \subset \mathcal{D}_{\underline{\Delta}}$. We abbreviate $\mathfrak{R}_{\sigma} := \mathcal{D}_{\delta}(\cup \mathcal{R}_{\sigma})$. Let $\mathcal{W}$ be an arbitrary cover of $\mathcal{Q} \cap \mathfrak{R}_{\sigma}$ by dyadic cubes with side $\leq 1$. We may (again) assume that the elements in $\mathcal{W}$ are disjoint and have side-length $\geq \delta$. 

As in the proof of \eqref{form3app}, let $\mathcal{Q}_{0} \subset \mathcal{Q}$ be the squares in $\mathcal{Q}$ which are not strictly contained in any element of $\mathcal{W}$. Our claim is formally the same as \eqref{form4}, namely
\begin{equation}\label{form10} \mathop{\sum_{W \in \mathcal{W}}}_{W \subset Q} \ell(W)^{t - \epsilon} \geq \ell(Q), \qquad Q \in \mathcal{Q}_{0}, \, t \in [1 + \epsilon,\tfrac{3}{2}]. \end{equation}
This implies $\mathcal{H}^{t - \epsilon}_{\infty}(\mathcal{Q} \cap \mathfrak{R}_{\sigma}) \geq \mathcal{H}^{t}_{\infty}(\mathcal{Q})$ by repeating the argument below \eqref{form4}.

Let us prove \eqref{form10}. Fix $Q \in \mathcal{Q}_{0} \subset \mathcal{D}_{\underline{\Delta}}$. We first record some properties of $Q \cap \mathfrak{R}_{\sigma}$. We already discussed earlier that the tubes $T \in \mathcal{T}_{\sigma}$ are $\sim \delta^{1/2}$-separated. This implies the following non-concentration condition for $\mathfrak{R}_{\sigma}$:
\begin{equation}\label{form11} |\mathfrak{R}_{\sigma} \cap W| \lesssim \begin{cases} \ell(W)^{2}\delta^{-3/2}, & \delta^{1/2} \leq \ell(W) \leq \underline{\Delta}, \\ \ell(W)\delta^{-1}, & \delta \leq \ell(W) \leq \delta^{1/2}, \end{cases} \qquad W \in \mathcal{D}, \, \sigma \in \delta \Z \cap [0,1].\end{equation} 
We also need a lower bound for $|\mathfrak{R}_{\sigma} \cap Q|$. For this purpose, we start by noting that if $\sigma \in \delta \Z \cap [0,1]$, then many elements of $\mathcal{T}_{\sigma}$ intersect $Q$ significantly. More precisely, there exists an interval $I_{Q} \subset [-2,2]$ of length $\ell(I) \sim \underline{\Delta}$ such that $T_{\sigma,b}$ "enters $Q$ from the left and exits from the right" for all $b \in I_{Q}$, see Figure \ref{fig4}. A fully rigorous definition is that $T_{\sigma,b}$ intersects $Q$ but does not intersect the top and bottom segments in $\partial Q$. In particular, 
\begin{equation}\label{form16} b \in I_{Q} \quad \Longrightarrow \quad |T_{\sigma,b} \cap Q|_{\delta} \sim \underline{\Delta}/\delta, \end{equation} 
where (as before) $T_{\sigma,b} = [\ell_{\sigma,b}]_{\delta/2}$.

Now, we apply the "density property" of $\mathcal{P}_{\sigma}$ stated in Lemma \ref{lemma1}\nref{P1} to $I_{Q}$:
\begin{displaymath} |Y_{\sigma} \cap I_{Q}| \sim \underline{\Delta} \delta^{-1/2}, \qquad \sigma \in [0,1]. \end{displaymath} Now all the tubes $T = T_{b,\sigma} \in \mathcal{T}_{\sigma}$ with $b \in Y_{\sigma} \cap I_{Q}$ intersect $Q$ significantly. However, this is not enough to say anything about $Q \cap \mathfrak{R}_{\sigma}$, because $\mathfrak{R}_{\sigma} = \mathcal{D}_{\delta}(\cup \mathcal{R}_{\sigma})$ is defined as the union of the $(\Delta \times \delta)$-rectangles $R_{T} \subset T$, not the full tubes. To get around this, we have to recall from \eqref{def:rectangles} how the rectangles $R_{T} \subset T$ were selected.
\begin{figure}[h!]
\begin{center}
\begin{overpic}[scale = 0.5]{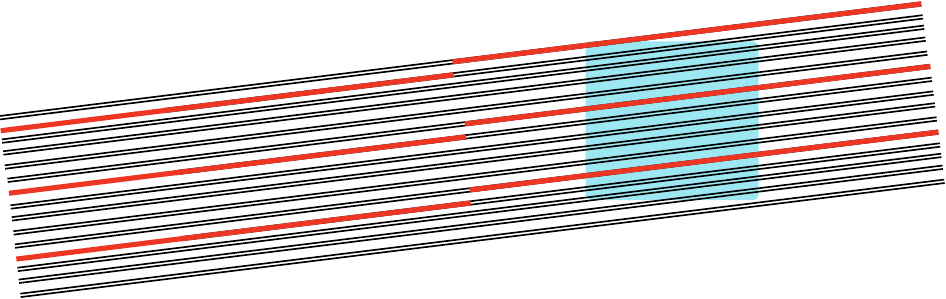}
\end{overpic}
\caption{The square $Q \in \mathcal{D}_{\underline{\Delta}}$ (drawn in blue) intersects many rectangles $R_{T}$, $T \in \mathcal{T}_{\sigma}$ ( drawn in red).}\label{fig4}
\end{center}
\end{figure}

Recall the (increasing) enumeration $Y_{\sigma} = \{b_{1},\ldots,b_{n}\}$. Fix $b_{j} \in Y_{\sigma} \cap I_{Q}$ (so "$T_{\sigma,b}$ enters $Q$ from the left and exits from the right"), and let $R_{\sigma,b_{j}} \subset T_{\sigma,b_{j}}$ be the (unique) element of $\mathcal{R}_{\sigma}$ contained in $T_{\sigma,b_{j}}$. Recall from \eqref{def:rectangles} that 
\begin{displaymath} R_{\sigma,b_{j}} = T_{\sigma,b_{j}} \cap \bar{\pi}_{\sigma}^{-1}(2I_{i(j)}) \end{displaymath}
for the interval $I_{i(j)} \in \{I_{1},\ldots,I_{m}\} := \mathcal{D}_{\Delta/2}([-2,2])$ satisfying $j = i(j) (\mathrm{mod\, } m)$. 

Next note that the (hypothetical) inclusion $\bar{\pi}_{\sigma}(Q) \subset 2I_{i(j)}$ implies $R_{\sigma,b_{j}} \cap Q = T_{\sigma,b_{j}} \cap Q$, and therefore $R_{\sigma,b_{j}}$ intersects $Q$ significantly -- more precisely $|R_{\sigma,b_{j}} \cap Q|_{\delta} \sim \underline{\Delta}/\delta$. On the other hand, since $\ell(Q) = \underline{\Delta} \leq \tfrac{1}{100}\Delta$, there exists at least one index $i_{Q} \in \{1,\ldots,m\}$ such that $\bar{\pi}_{\sigma}(Q) \subset 2I_{i_{Q}}$. Summarising these two observations, we conclude the following: whenever $b_{j} \in Y_{\sigma} \cap I_{Q}$ is such that $i(j) = i_{Q}$, then $|R_{\sigma,b_{j}} \cap Q|_{\delta} \sim \underline{\Delta}/\delta$.

Are there (m)any elements $b_{j} \in Y_{\sigma} \cap I_{Q}$ such that $i(j) = i_{Q}$? By the definition of "$i(j)$", this requirement is equivalent to $j = i_{Q} (\mathrm{mod\,} m)$. Finally, note that $Y_{\sigma} \cap I_{Q} = \{b_{k},\ldots,b_{l}\}$ for some $k,l \in \{1,\ldots,n\}$ with $l - k \sim |Y_{\sigma} \cap I_{Q}| \sim \underline{\Delta}\delta^{-1/2}$. Evidently,
\begin{displaymath} |\{b_{j} \in Y_{\sigma} \cap I_{Q} : j = i_{Q}(\mathrm{mod\,} m)\}| \sim (l - k)/m \sim \Delta\underline{\Delta} \delta^{-1/2} \stackrel{\eqref{form9}}{\geq} \underline{\Delta}^{1 + \epsilon/2}\delta^{-1/2}. \end{displaymath}
Therefore,
\begin{displaymath} |\{b_{j} \in Y_{\sigma} \cap I_{Q} : |R_{\sigma,b_{j}} \cap Q|_{\delta} \sim \underline{\Delta}/\delta\}| \gtrsim \underline{\Delta}^{1 + \epsilon/2}\delta^{-1/2}, \end{displaymath}
and finally
\begin{equation}\label{form25} |\mathfrak{R}_{\sigma} \cap Q| \gtrsim |\{b_{j} \in Y_{\sigma} \cap I_{Q} : |R_{\sigma,b_{j}} \cap Q|_{\delta} \sim \underline{\Delta}/\delta\}| \cdot (\underline{\Delta}/\delta) \gtrsim \underline{\Delta}^{2 + \epsilon/2}\delta^{-3/2}. \end{equation} 
The upper bound \eqref{form11} and the lower bound above are the same as in \eqref{form8}-\eqref{form12}, except that $s = (1 - \epsilon)/2$ has been replaced by $1/2$, and the lower bound has the additional (small) factor $\underline{\Delta}^{\epsilon/2}$. Now virtually the same argument as in the proof of \eqref{form4} yields \eqref{form10}. We repeat the details for completeness. First, assume that at least half of the squares in $\mathfrak{R}_{\sigma} \cap Q$ are contained in such squares $W \in \mathcal{W}$ with $\delta^{1/2} \leq \ell(W) \leq \underline{\Delta}$. Then,
\begin{align*} \mathop{\sum_{W \in \mathcal{W}}}_{W \subset Q} \ell(W)^{t - \epsilon} & \geq \underline{\Delta}^{t - \epsilon - 2}\delta^{3/2} \mathop{\sum_{W \in \mathcal{W}}}_{W \subset Q, \delta^{1/2} \leq \ell(W) \leq \underline{\Delta}} \ell(W)^{2}\delta^{-3/2} \stackrel{\eqref{form11} + \eqref{form25}}{\gtrsim} \underline{\Delta}^{t - \epsilon/2}. \end{align*} 
This yields \eqref{form10}, using \eqref{form9}, and recalling $\ell(Q) = \underline{\Delta}$.

Finally, assume that at least half of the squares in $\mathfrak{R}_{\sigma} \cap Q$ are contained in squares $W \in \mathcal{W}$ with $\delta \leq \ell(W) \leq \delta^{1/2}$. Then, recalling that $t \in [1 + \epsilon,\tfrac{3}{2}]$, 
\begin{align*} \mathop{\sum_{W \in \mathcal{W}}}_{W \subset Q} \ell(W)^{t - \epsilon} & \geq \delta^{t - \epsilon} \mathop{\sum_{W \in \mathcal{W}}}_{W \subset Q, \delta \leq \ell(W) \leq \delta^{1/2}} \ell(W)\delta^{-1}\\
& \stackrel{\eqref{form11} + \eqref{form25}}{\gtrsim} \underline{\Delta}^{2 + \epsilon/2}\delta^{t - \epsilon - 3/2} \geq \underline{\Delta}^{2 + \epsilon/2}\delta^{-\epsilon/2}. \end{align*}
The right hand side is $\geq C\underline{\Delta}^{t}$ by \eqref{form7}. This completes the proof of \eqref{form10}, and the proof of Proposition \ref{prop1App}. \end{proof}

\bibliographystyle{plain}
\bibliography{references}

\end{document}